\documentclass[11pt, reqno, a4paper]{amsart}

\usepackage{hyperref, amssymb, amsmath, enumerate, amsfonts, amsthm, dsfont, mathrsfs, url, bm}

\advance \topmargin by -\headheight
\advance \topmargin by -\headsep

\evensidemargin \oddsidemargin
\newtheorem{theorem}{Theorem}[section] 
\newtheorem{lemma}[theorem]{Lemma}     

  \theoremstyle{remark}
\newtheorem{remark}[theorem]{Remark}
\newtheorem*{acknowledgements}{Acknowledgements}

\theoremstyle{definition}
\newtheorem{definition}[theorem]{Definition}

\theoremstyle{claim}
\newtheorem*{claim}{Claim}

\numberwithin{equation}{section}          

 \newcommand{\set}[1]{\left\{#1\right\}}
\newcommand{\bigset}[1]{\bigl\{ #1 \bigr\}}
\newcommand{\Bigset}[1]{\Bigl\{ #1 \Bigr\}}

\newcommand{\Oh}[1]{ O\left( #1\right)}
\newcommand{\abs}[1]{\left| #1\right|}
\newcommand{\bigabs}[1]{\bigl| #1 \bigr|}
\newcommand{\Bigabs}[1]{\Bigl| #1 \Bigr|}
\newcommand{\biggabs}[1]{\biggl| #1 \biggr|}

\newcommand{\sqbrac}[1]{\left[ #1 \right]}
\newcommand{\ceil}[1]{\left\lceil #1 \right\rceil}

\newcommand{\floor}[1]{\left\lfloor #1 \right\rfloor}
\newcommand{\brac}[1]{\left( #1 \right)}
\newcommand{\bigbrac}[1]{\bigl( #1 \bigr)}
\newcommand{\Bigbrac}[1]{\Bigl( #1 \Bigr)}
\newcommand{\biggbrac}[1]{\biggl( #1 \biggr)}
\newcommand{\Biggbrac}[1]{\Biggl( #1 \Biggr)}
\newcommand{\norm}[1]{\left\| #1\right\|}
\newcommand{\bignorm}[1]{\big\| #1 \big\|}

\newcommand{\recip}[1]{\frac{1}{#1}}
\newcommand{\trecip}[1]{\tfrac{1}{#1}}

\newcommand{\ve}{\mathbf{e}}
\newcommand{\vx}{\mathbf{x}}
\newcommand{\vy}{\mathbf{y}}
\newcommand{\valpha}{\bm{\alpha}}
\newcommand{\vbeta}{\bm{\beta}}

\newcommand{\vgamma}{\bm{\gamma}}

\newcommand{\vX}{\mathbf{X}}
\newcommand{\vP}{\mathbf{P}}

\newcommand{\va}{\mathbf{a}}

\newcommand{\vb}{\mathbf{b}}
\newcommand{\vc}{\mathbf{c}}
\newcommand{\vw}{\mathbf{w}}
\newcommand{\vz}{\mathbf{z}}
\newcommand{\vt}{\mathbf{t}}
\newcommand{\vn}{\mathbf{n}}
\newcommand{\vm}{\mathbf{m}}

\newcommand{\vr}{\mathbf{r}}

\newcommand{\vi}{\mathbf{i}}

\newcommand{\vs}{\mathbf{s}}
\newcommand{\vxi}{\bm{\xi}}
\newcommand{\vF}{\mathbf{F}}

\newcommand{\vT}{\mathbf{T}}

\newcommand{\vzeta}{\bm{\zeta}}
\newcommand{\veta}{\bm{\eta}}

\newcommand{\vsigma}{\bm{\sigma}}

\newcommand{\M}{\mathfrak{M}}
\newcommand{\m}{\mathfrak{m}}

\newcommand{\twosum}[2]{ \sum_{\substack{#1\\ #2}}}

\newcommand{\N}{\mathbb{N}}
\newcommand{\Z}{\mathbb{Z}}
\newcommand{\Q}{\mathbb{Q}}
\newcommand{\R}{\mathbb{R}}
\newcommand{\C}{\mathbb{C}}
\newcommand{\T}{\mathbb{T}}
\newcommand{\F}{\mathbb{F}}

\newcommand{\Jac}{\mathrm{Jac}}
\newcommand{\spn}{\mathrm{span}}

\newcommand{\adj}{\mathrm{adj}}

\newcommand{\meas}{\mathrm{meas}}

\newcommand{\intd}{\mathrm{d}}

\newcommand{\eps}{\varepsilon}
\newcommand{\hash}{\#}

\begin{document}

\title[Solution-free sets]{Solution-free sets for sums of binary forms}
\author{Sean Prendiville }

\thanks{During the completion of this work, the author was supported by an EPSRC doctoral training grant through the University of Bristol.}
\address{School of Mathematics, University of Bristol, University Walk, Clifton, 
Bristol BS8 1TW, United Kingdom}
\email{sean.prendiville@bristol.ac.uk}

\maketitle
\begin{abstract}
In this paper we obtain quantitative estimates for the asymptotic density of subsets of the integer lattice $\Z^2$ which contain only trivial solutions to an additive equation involving binary forms.  In the process we develop an analogue of Vinogradov's mean value theorem applicable to binary forms.
\end{abstract}
\section{Introduction}

Certain systems of linear equations have the property that, should a set of integers fail to deliver non-trivial solutions to the system, then the set has zero density.  The problem of obtaining quantitative asymptotic estimates for the density of such sets, first addressed successfully by Roth \cite{roth53, roth53II}, is one which has seen remarkable advances over the last decade;  spectacularly in the work of Gowers \cite{gowers01} on long arithmetic progressions, Bourgain \cite{bourgain99, bourgain08} on progressions of length three and Green and Tao \cite{greentao09} on progressions of length four.  Recently, M. L. Smith \cite{smith09} has obtained density estimates for sets of integers which do not contain solutions to a  class of homogeneous equations involving $k$th powers.  This was the first general result on inherently non-linear systems.  In this paper we not only generalise Smith's result from an equation involving $k$th powers to one involving binary forms, but we also extract density estimates for subsets of the two-dimensional integer lattice.  Our approach uses the density increment method of Roth and Gowers, together with the circle method.  A notable feature in our application of the circle method is a novel analogue of Vinogradov's mean value theorem, applicable to systems of equations involving binary forms.  Our approach to this mean value theorem makes intrinsic use of the structure of the shift-invariant system associated with our equation, and thereby improves on those estimates which can be deduced from the much more general work of Parsell \cite{parsell05} on multi-dimensional versions of Vinogradov's mean value theorem. 

In order to describe our conclusions, we first introduce some notation.  When $\Phi \in \Z[x,y]$ is a  binary form we write $\Phi^{u,v}$ for the derivative $\frac{\partial^{u+v}}{\partial x^u \partial y^v}\Phi(x,y)$.

\begin{definition}  Let us say the tuple $\vc = (c_1, \dots, c_s)$ of non-zero integers is a \emph{non-singular choice of coefficients} for $\Phi$ if there exist binary forms $\Phi_1, \dots, \Phi_N$ satisfying
\begin{equation}
\begin{split}
\set{ \Phi_1, \dots, \Phi_N} &\subset \set{\Phi^{u,v} : 0 \leq u+v < k}\\ & \subset\ \spn\set{ \Phi_1, \dots, \Phi_N},
\end{split}
\end{equation} 
such that the auxiliary system of equations
\begin{equation}
 c_1 \Phi_i(\vx_1) + \dots + c_s \Phi_i(\vx_s) = 0 \quad (1\leq i \leq N),
\end{equation}
has non-singular\footnote{Here \emph{non-singular} means the associated Jacobian has full-rank over the field in question.} real and $p$-adic solutions for every prime $p$.
\end{definition}

\begin{definition} We call a $2s$-tuple $(\vx_1, \dots, \vx_s)$ \emph{diagonal} if there exists an affine line $L = \va + \R \cdot \vb$ such that $\vx_i \in L$ for all $i$.
\end{definition}

Writing $[X]$ for the set $\set{1, 2, \dots, \floor{X}}$, the most accessible of our density results can now be stated.

\begin{theorem}\label{intro result 1} Let $\Phi \in \Z[x,y]$ be a binary form of degree $k\geq 2$ and let $\vc \in \Z^s$ be a non-singular choice of coefficients for $\Phi$, with $c_1 + \dots + c_s = 0$ .  Suppose that $s \geq \tfrac{3}{4}k^3 \log k(1 + o(1))$.  Then any set $A \subset [X]^2$ containing only diagonal solutions to the equation 
\begin{equation}\label{single equation}
 c_1 \Phi(\vx_1) + \dots + c_s \Phi(\vx_s) = 0 \qquad (\vx_i \in A),
\end{equation}
satisfies the bound
\begin{equation}\label{density bound 1}
 |A| \ll X^2\brac{\log\log X}^{-1/(s-1)},
\end{equation}
where the implicit constant depends only on $\vc$ and $\Phi$.
\end{theorem}

\begin{remark}

For a more precise lower bound on the number of variables required than $s \geq \tfrac{3}{4} k^3 \log k(1 + o(1))$, see Theorem \ref{full density theorem}. 
\end{remark}

For comparison, recent work of Smith \cite{smiththesis} establishes a version of the above result in which $\Phi$ is replaced by a $k$th power and the set $A$ is a subset of the integers in the interval $[1, X]$.  Indeed, our insistence that $A$ contains only diagonal solutions to \eqref{single equation} precludes the deduction of Theorem \ref{intro result 1} from Smith's result.  We also note that Smith obtains an exponent of $\log\log N$ in \eqref{density bound 1} of the form $-2^{-2^k}$.  

One can obtain a qualitative version of Theorem \ref{intro result 1} by applying the multidimensional Szemer\'{e}di theorem of Furstenberg and Katznelson \cite{furstenbergkatznelson78}.  In this way, one can show that any (infinite) set $A \subset \Z^2$ containing only diagonal solutions to \eqref{single equation} must have zero upper Banach density.  If one had a quantitative version of the multi-dimensional Szemer\'{e}di theorem providing bounds analogous to the one-dimensional bounds of Gowers \cite{gowers01}, then one could use this result to obtain bounds of the form \eqref{density bound 1} in Theorem \ref{intro result 1}.  However, the exponent of $\log\log N$ in these bounds would be intrinsically dependent on the choice of form $\Phi$ and coefficients $c_1, \dots, c_s$, whereas our result depends only on $s$.  Moreover, no such two-dimensional bounds currently exist; the best bounds presently available are due to Shkredov \cite{shkredov05} and are not general enough for our purposes.

To obtain Theorem \ref{intro result 1}, we bound the density of sets which contain only diagonal solutions to the larger system of equations
\begin{equation}\label{non-singular solution equation 2}
c_1\Phi^{u,v}(\vx_1) + \dots + c_s\Phi^{u,v}(\vx_s) = 0 \qquad (0 \leq u+v < k).
\end{equation}
Sets avoiding non-diagonal solutions to this larger system may have greater size than those avoiding non-diagonal solutions to \eqref{single equation}. However, a key observation is that this larger system enjoys \emph{translation-dilation invariance}, in that $(\vx_1, \dots, \vx_s)$ satisfies \eqref{non-singular solution equation 2} if and only if $(\lambda\vx_1 + \vxi, \dots, \lambda \vx_s + \vxi)$ satisfies \eqref{non-singular solution equation 2}, whenever $\lambda \neq 0$.  This invariance allows us to adapt the density increment method of Roth and Gowers \cite{roth53, gowers01}.

In order to implement the density increment method it is necessary to have an asymptotic estimate for the number of solutions to \eqref{non-singular solution equation 2} with variables restricted to the interval $[1, X]$.  This we obtain through an application of the Hardy--Littlewood method.  In order to deal with the minor arcs, we utilise Vinogradov's method\footnote{See Chapter 4 of \cite{montgomery94} for a description of this method.}, which necessitates the estimation of the number $J_{s, \Phi}(X)$ of solutions $(\vx, \vy) \in [X]^{4s}$ to the system of equations
\begin{equation}\label{dependent vino system}
\sum_{j=1}^s \Phi^{u,v}(\vx_j) = \sum_{j=1}^s \Phi^{u,v} (\vy_j)\quad (0 \leq u+v < k).
\end{equation}
When $\Phi$ takes the form  $a(b x + c y)^k$, such an estimate can be obtained from the standard Vinogradov mean value theorem, as found in \cite[Chapter 5]{vaughan97}.  We must therefore treat the remaining case.
 \begin{definition}  We say a binary form $\Phi \in \Z[x,y]$ of degree $k$ is \emph{degenerate} if it takes the form $(\alpha x + \beta y)^k$ for some $\alpha, \beta \in \C$.  One can check that $\Phi$ is degenerate if and only if there exist $a,b,c \in \Z$ such that $\Phi = a(bx + cy)^k$. 
 \end{definition}
 
 \begin{definition}
 We define the \emph{differential dimension} of $\Phi$ to be the dimension $N$ of the linear span of the set of non-constant derivatives
 \begin{equation}\label{set of derivatives}
 \set{\Phi^{u,v} : 0 \leq u+v < k}.
 \end{equation}
Given a maximal linearly independent subset $\set{F_1, \dots, F_N}$ of  \eqref{set of derivatives}, we define the \emph{differential degree} of $\Phi$  to be the quantity
 \begin{equation}
K = \sum_i \deg F_i.
\end{equation}
Elementary linear algebra confirms that $K$ is independent of our choice of $F_i$.
\end{definition}

Our mean value theorem for non-degenerate binary forms is then the following.

\begin{theorem}\label{intro VMVT} Let $\Phi \in \Z[x,y]$ be a non-degenerate binary form of degree $k$, differential dimension $N$ and differential degree $K$.  Write $M = \ceil{N/2}$, and define
\begin{equation}\label{exponent decay}
\Delta_s  = K\brac{1-\trecip{k}}^{\floor{s/M}}.
\end{equation} 
Then we have the bounds
\begin{equation}\label{mean value theorem bounds}
X^{4s - K}  \ll J_{s, \Phi}(X) \ll X^{4s - K + \Delta_s},
\end{equation}
where the implicit constants depend only on $s$ and $\Phi$.
\end{theorem}

We remark that when $\Phi$ is a degenerate binary form, then $K = k(k+1)/2$ and $N = k$.  Hence our result is comparable  to the standard Vinogradov mean value theorem, where one obtains
$$
\Delta_s \leq \trecip{2} k^2 \brac{1 - \trecip{k}}^{\floor{s/k}}.
$$  
Using very general work of Parsell \cite{parsell05}, one can extract a bound on the exponent $\Delta_s$ in Theorem \ref{intro VMVT} of the form 
$$
\Delta_s \leq rk\, e^{2- 2s/rk},
$$
where $r = (k+2)(k+3)/2 - 1$.  By way of comparison, an immediate consequence of Theorem \ref{intro VMVT} is the bound
$$
\Delta_s \leq Ke^{-\frac{1}{k}\floor{2s/(N+1)}},
$$ 
and one certainly has $K < rk$ and $N < r$.  Moreover, Parsell's general theorem is obtained through the somewhat formidable method of \emph{repeated efficient differencing}.  We are able to extract our result from the comparatively simple $p$-adic iterative method, originating with Linnik \cite{linnik43}, and reaching a refined state in work of Karatsuba \cite{karatsuba73} and Stechkin \cite{stechkin75}.

An expert in the field might hope to apply the above result via Vinogradov's method to obtain superior bounds for exponential sums over binary forms, at least when $k$ is large.  However, as demonstrated in Wooley \cite[\S 8]{wooley99}, one can already attain such bounds using the standard Vinogradov mean valued theorem.

\subsection{Notation}

Throughout the remainder of the paper we fix a non-degenerate binary form $\Phi$ of degree $k$, differential dimension $N$ and differential degree $K$.  We reserve the letter $M$ for the quantity $\ceil{N/2}$.  Let us also fix $\set{F_1, \dots, F_N}$, a maximal linearly independent subset of $\set{\Phi^{u,v} : 0 \leq u+ v < k}$.  Let $\vF$ denote the tuple $(F_1, \dots, F_N)$.  Setting $k_i = \deg F_i$, we always assume that   $k = k_1 \geq k_2 \geq \dots \geq k_N = 1$.   Using Taylor's formula, a convenient consequence of our ordering of the $F_i$ is that for any $\vxi \in \Z^2$ there exists a lower unitriangular\footnote{A lower triangular matrix with all diagonal entries equal to one.} matrix $\Xi_{\vxi} \in GL_N(\Q)$ such that
\begin{equation}\label{matrix translation invariance}
\vF(\vx + \vxi) = \Xi_{\vxi} \cdot \vF(\vx) + \vF(\vxi).
\end{equation}
We call this property \emph{translation-dilation invariance}, since it implies that for any $\vxi \in \R^2$ and $\lambda \neq 0$ we have the equivalence
\begin{equation}\label{equation translation invariance}
\sum_{j =1}^s \Bigbrac{\vF(\vx_j) - \vF(\vy_j)} = 0 \quad \Longleftrightarrow \quad \sum_{j =1}^s \Bigbrac{\vF(\lambda \vx_j +\vxi) - \vF(\lambda \vy_j+ \vxi)} = 0.
\end{equation}
Given a real $X \geq 1$ write $[X]$ for $\set{1, 2, \dots, \floor{X}}$.  We use $J_{s, \Phi}(X; \vm)$ to denote the number of $(\vx, \vy) \in [X]^{4s}$ satisfying
\begin{equation}
\sum_{j=1}^s \Bigbrac{\vF(\vx_j) - \vF(\vy_j)} = \vm.
\end{equation}
Notice that $J_{s, \Phi}(X; \mathbf{0})$ coincides with our definition of $J_{s, \Phi}(X)$.

We analyse both the equations \eqref{non-singular solution equation 2} and \eqref{dependent vino system} via the exponential sum
\begin{equation}\label{f sum defn}
f(\valpha) = f(\valpha; X) = \sum_{\vx \in [X]^2} e(\valpha\cdot \vF(\vx)),
\end{equation}
where $e(y) = e^{2\pi i y}$.  By the orthogonality relations we have
\begin{equation}\label{orthogonality relations}
J_{s, \Phi}(X; \vm) = \oint |f(\valpha)|^{2s}e(\valpha\cdot \vm) \intd\valpha,
\end{equation}
where $\oint$ denotes the integral over the $N$-dimensional torus $\T^N = \R^N/\Z^N$.

Throughout, we assume that $X$ is sufficiently large in terms of $s$, $\vc$ and $\vF$, so all implicit constants depend only on these parameters, unless otherwise indicated.  We note that $\vF$ depends ultimately only on $\Phi$.

\section{The Mean Value Theorem}\label{mean value section}

Before working towards upper bounds for $J_{s,\Phi}(X)$, let us derive an elementary lower bound.  By \eqref{orthogonality relations}, for any $\vm$ we have  $J_{s, \Phi}(X; \vm) \leq J_{s, \Phi}(X)$.  Notice that there are $O_{\vF,s}(X^K)$ values of $\vm$ for which $J_{s,\Phi}(X; \vm)$ is non-zero.  Summing over these values, we obtain
\begin{equation}\label{lower bound proof}
X^K J_{s,\Phi}(X) \gg X^{4s}.
\end{equation}
The lower bound in \eqref{mean value theorem bounds} follows.

The remainder of this section is occupied with proving the upper bound in \eqref{mean value theorem bounds}.  We expect the majority of solutions to  \eqref{dependent vino system} to be non-singular (in a sense to be defined later), whilst the remaining set of singular solutions should be relatively sparse.  To define the appropriate notion of singularity neccesitates the discussion of the Jacobian associated to \eqref{dependent vino system}.  

\begin{definition}
Write $\Jac (\vx_1, \dots, \vx_M)$ for the $N\times 2M$ matrix 
\begin{equation}\label{Jac defn}
\Bigbrac{ F_i^{1,0}(\vx_j),\ F_i^{0,1}(\vx_j)}_{\substack{ 1\leq i \leq N\\ 1 \leq j \leq M}},
\end{equation}
and let $\Delta(\vx_1, \dots, \vx_M)$ denote the determinant of the $N \times N$ matrix consisting of the first $N$ columns of $\Jac(\vx_1, \dots, \vx_M)$. 
\end{definition} In order to get our version of Linnik's $p$-adic iterative method to work, $\Delta$ cannot be identically zero.  Notice that if $\Phi$ is degenerate, then $\Delta$ \emph{is} identically zero.  Our first lemma, Lemma \ref{no one dim space}, feeds into our second, Lemma \ref{non-zero jacobian}, which establishes that $\Delta$ is non-zero when and only when $\Phi$ is non-degenerate.  We keep Lemma \ref{no one dim space} separate as it proves useful later.

\begin{lemma}\label{no one dim space}
Suppose there exists $1 \leq l < \deg \Phi$ such that the linear span of the set $\set{\Phi^{u,v}: u + v =  l}$ is one-dimensional.  Then $\Phi$ is degenerate.
\end{lemma}

\begin{proof}
For any binary form $\Phi$ of degree $k$ and $0 \leq l \leq k$ one can show by induction that
\begin{equation}\label{derivative identity}
\Phi = \frac{(k-l)!}{k!} \sum_{r = 0}^{l} \binom{l}{r} x^{l-r}y^r \Phi^{l-r, r}.
\end{equation}
Let $F_i$ be the only form from $F_1, \dots, F_N$ with degree $k-l$.  Then for each $r  = 0, \dots, l$ there exists $\lambda_r\in \Q$ such that $\Phi^{l-r, r} = \lambda_r F_i$.  We must have $\lambda_r \neq 0$ for some $r$.  Let us suppose that $r >0$, the case $r< l$ being similar.  We have 
$$
F_i^{1,0} = \lambda_r^{-1} \Phi^{l-r+1, r} = \lambda_r^{-1}\lambda_{r-1} F_i^{0,1}.
$$
Letting $\lambda = \lambda_r^{-1}\lambda_{r-1}$ and iterating one sees that for all $0\leq s \leq k_i$ we have $F_i^{k_i -s, s} = \lambda^{k_i -s} F_i^{0, k_i}$.  Using this and \eqref{derivative identity}, it follows that
\begin{align*}
F_i & = \frac{F_i^{0, k_i}}{k_i!} \sum_{r =0}^{k_i} \binom{k_i}{r} \lambda^{k_i - r} x^{k_i -r } y^r \\
	& =  (\alpha x + \beta y)^{k_i}
\end{align*}
for some real $\alpha$ and $\beta$, with $\beta \neq 0$.  Differentiating in $y$ we see that 
$$
\Phi^{l -r , r+1} = \lambda_r F_i^{0,1} = \lambda_r \beta k_i (\alpha x + \beta y)^{k_i -1}.
$$
Differentiating in $x$ when $r < l$ we also see that
$$
\Phi^{l-r, r+1}  = \lambda_{r+1} F_i^{1,0} = \lambda_{r+1} \alpha k_i (\alpha x + \beta y)^{k_i -1}.
$$
Thus for each $r=0,1, \dots, l$, one obtains
$$
\lambda_r = (\alpha/\beta) \lambda_{r+1} = \dots = (\alpha/\beta)^{l-r} \lambda_l.
$$
Inputting this into \eqref{derivative identity} we deduce that
\begin{align*}
\Phi & = \frac{l!}{k!} \sum_{r = 0}^{l} \binom{l}{r} x^{l-r}y^r (\alpha/\beta)^{l-r} \lambda_l (\alpha x+\beta y)^{k-l}\\
& =  \frac{l!\lambda_l}{k!\beta^{l}} \brac{\alpha x+\beta y}^k.
\end{align*}
Therefore $\Phi$ is degenerate.
\end{proof}


\begin{lemma}\label{non-zero jacobian} If $\Phi$ is a non-degenerate binary form, then the determinant $\Delta$ is not the zero polynomial.
\end{lemma}

\begin{proof}
For each $1 \leq l \leq k$, let $I(l)$ denote the set of indices $i$ for which $k_i := \deg F_i = l$.  For each $i \in I(l)$ there exists $\vc_{li} \in \Z^{l+1}$ such that $F_i(x, y) = \vc_{li} \cdot (x^l, x^{l-1}y, \dots, y^l)$.  Let $C_l$ denote the matrix whose rows comprise $\vc_{li}$ ($i \in I(l)$).  Since the $F_i$ are linearly independent, $C_l$ has full-rank.  Hence there exists an invertible matrix $B_l$ such that $B_l C_l$ is a full-rank matrix in reduced row-echelon form.  Define the  rational homogeneous polynomials $G_1, \dots, G_N$ by 
$$
\begin{pmatrix} G_1\\ \vdots \\ G_N \end{pmatrix} = \begin{pmatrix} B_k & \ & \ & \\ \ & B_{k-1} & \ & \ \\ \ & \ &  \ddots \ & \\ \ & \ & \ &\  B_1 \end{pmatrix}  \cdot\begin{pmatrix} F_1\\ \vdots  \\ F_N \end{pmatrix}.
$$ 
From our construction, we see that $\deg G_i = \deg F_i$ for all $i$.  Furthermore, if $d_i$ denotes the highest exponent of $x$ occurring in $G_i(x,y)$, then for any $i, j \in I(l)$ with $i < j$ we have $d_i > d_j$.  Write $\widetilde{\Jac}(\vx_1, \dots, \vx_M)$ for the $N\times 2M$ matrix 
\begin{equation}
\Bigbrac{ G_i^{1,0}(\vx_j),\ G_i^{0,1}(\vx_j)}_{\substack{ 1\leq i \leq N\\ 1 \leq j \leq M}},
\end{equation}
and let $\tilde{\Delta}(\vx_1, \dots, \vx_M)$ denote the determinant of its first $N$ columns.  By linearity of differentiation, we have
$$
\widetilde{\Jac}(\vx_1, \dots, \vx_M) = \begin{pmatrix} B_k & \ & \ & \\ \ & B_{k-1} & \ & \ \\ \ & \ &  \ddots \ & \\ \ & \ & \ &\  B_1 \end{pmatrix}  \cdot \Jac(\vx_1, \dots, \vx_M).
$$
Since the matrix with the $B_i$ along the diagonal is non-singular, it suffices to prove that $\tilde{\Delta}(\vx_1, \dots, \vx_M)$ is not the zero polynomial.  

For $r$ in the range $1 \leq r \leq \frac{N}{2}$, define $D_r(\vx_1, \dots, \vx_r)$ to be the determinant of the $2r\times 2r$ matrix occurring in the bottom-left corner of $\widetilde{\Jac}(\vx_1, \dots, \vx_M)$.  We induct on $r$ to show $D_r$ is not the zero polynomial.  When $N = 2M$ this completes the proof, since in this case $D_M = \tilde{\Delta}$.  When $N + 1= 2M$ we expand along the $N$th column of the matrix associated to $\tilde{\Delta}$ to obtain
\begin{equation}\label{odd case one column}
\tilde{\Delta}  =  D_{M-1}(\vx_1, \dots, \vx_{M-1})G_1^{1,0}(\vx_M) + \sum_{i=2}^N P_i(\vx_1, \dots, \vx_{M-1})G_i^{1,0}(\vx_M),
\end{equation}
for some polynomials $P_2, \dots, P_N$.  Since $F_1$ is the only form of degree $k_1 = k$ and is non-degenerate, $G_1(x,y)$ does not take the form $ cy^{k_1}$.  It follows that $G_1^{1,0}(x,y)$ is a non-zero polynomial of degree $k -1$, a degree higher than that of any other $G_i^{1,0}(x,y)$.  Since $D_{M-1}(\vx_1, \dots, \vx_{M-1})$ is also a non-zero polynomial, we can use \eqref{odd case one column} to compare the exponents of the monomials in $\tilde{\Delta}$ which feature $\vx_M$, and thereby deduce that $\tilde{\Delta}$ cannot be zero.

It remains to show that $D_r$ is non-zero for each $1\leq r \leq N/2$.  We begin with a claim.
\begin{claim}\label{wronskian claim}
For $2 \leq i \leq N$ the polynomial
\begin{equation}
W_i(x,y) = \begin{vmatrix}  G_{i-1}^{1,0} & G_{i-1}^{0,1} \\ G_i^{1,0} & G_i^{0,1} \end{vmatrix}  = G_{i-1}^{1,0}\ G_i^{0,1} - G_{i-1}^{0,1}\ G_i^{1,0}
\end{equation}
is non-zero, of degree $k_{i-1} + k_i -2$ and with highest exponent of $x$ equal to $d_{i-1} + d_i -1$.
\end{claim}  
Recalling that $d_i$ denotes the highest exponent of $x$ occurring in $G_i(x,y)$, consider the polynomial
\begin{align*}
 \begin{vmatrix}  d_{i-1} x^{d_{i-1} -1}y^{k_{i-1} - d_{i-1}} & (k_{i-1}- d_{i-1})  x^{d_{i-1} }y^{k_{i-1} - d_{i-1}-1} \\ d_{i} x^{d_i -1}y^{k_{i} - d_{i}} & (k_{i}- d_{i})  x^{d_i }y^{k_{i} - d_{i}-1} \end{vmatrix}& \\
 =  \bigbrac{d_{i-1}(k_i - d_i) - d_i(k_{i-1} - d_{i-1})}&x^{d_{i-1} + d_i -1}  y^{k_{i-1} + k_i - d_{i-1} - d_i -1} .
\end{align*}
If this is non-zero then, by our construction of the $G_j$, it has the same leading monomial and coefficient as $W_i$ (when we order monomials according to the lexicographical\footnote{So $(a_1, \dots, a_n) \prec (b_1, \dots, b_n)$ if there exists $i$ with $a_i < b_i$ and $a_j = b_j$ for all $j  < i$.}ordering on their exponents).  To establish the claim it therefore remains to show that
\begin{equation}\label{claim 1 follows if}
 \bigbrac{d_{i-1}(k_i - d_i) - d_i(k_{i-1} - d_{i-1})} \neq 0.
\end{equation}
Suppose otherwise.  Then \begin{equation}\label{di ki relation} d_{i-1} k_i = d_i k_{i-1}.\end{equation}  There are two cases to consider.  In the first case $k_i = k_{i-1}$, from which it follows that $d_{i-1} = d_i$.  However, this contradicts our construction of the $G_j$.  The only other possibility is that $k_{i-1} = k_i +1$.  In this case $k_{i-1}$ and $k_i$ are co-prime, so we must have $d_{i-1} = k_{i-1}$ and $d_i = k_i$.  Our construction of the $G_j$ therefore ensures that $G_{i-1}$ is the only $G_j$ of degree $k_{i-1}$, since it has the highest index of any $G_j$ of degree $k_{i-1}$, but also has highest exponent of $x$ equal to $k_{i-1}$.  This forces $\Phi$ to be degenerate, by Lemma \ref{no one dim space}, a contradiction which establishes the claim.

Notice that $D_1(\vx_1) = W_N(\vx_1)$, so $D_1$ is a non-zero polynomial by the claim, giving us the basis case of our induction.  Let us suppose that $D_{r-1}$ is non-zero, with $ r \leq N/2$.  Inspection reveals that $D_{r}$ is equal to 
\begin{equation}\label{inductive D identity}
\sum_{N-2r < i<j \leq N}  P_{ij}(\vx_1, \dots, \vx_{r-1}) \bigbrac{G_i^{1,0}(\vx_r)G_j^{0,1}(\vx_r) - G_i^{0,1}(\vx_r) G_j^{0,1}(\vx_r)},
\end{equation}
where the $P_{ij}$ are polynomials with $P_{ij} = D_{r-1}$ when $$\set{i,j} = \set{N -2r+2, N - 2r + 1}.$$  Let $W_{ij}$ denote the polynomial $G_i^{1,0}G_j^{0,1} - G_i^{0,1} G_j^{0,1}$.  The degree of the $W_{ij}$ occurring in \eqref{inductive D identity} is maximised only when $k_i = k_{N - 2r+1}$ and $k_j = k_{N-2r +2}$.  In this case, the highest exponent of $x$ occurring in $W_{ij}$ is strictly less than $d_{N-2r+1} + d_{N-2r+2} - 1$, unless $i = N-2r+1$ and $j = N-2r+2$, in which case $W_{ij} = W_{N-2r +2}$.  It follows from the claim and the induction hypothesis that the term
$$
P_{N-2r+1, N- 2r+2}(\vx_1, \dots, \vx_{r-1})W_{N-2r +2}(\vx_r)
$$
has a monomial occurring in no other term of the sum \eqref{inductive D identity}, hence $D_r$ is itself non-zero.  The lemma now follows. 
\end{proof}

The $p$-adic iterative method yields a congruence relation amongst the variables of equation \eqref{dependent vino system}.  In order to use this relation to provide an iterative bound on $J_{s, \Phi}(X)$, we need to count the number of solutions to such a congruence.  This is the purpose of the next lemma. 

\begin{definition}
Given $\vsigma \in \set{-1, 1}^M$, $\vm \in \Z^N$, $\vxi \in \Z^2$ and a prime $p$,  define $\mathcal{B}^{\vsigma}_p(\vm; \vxi)$ to be the set of solutions $(\vx_1, \dots, \vx_M)$ modulo $p^k$ of the system of equations
\begin{equation}\label{congruence condition equation}
\sum_{j =1}^M \sigma_j F_i(\vx_j - \vxi) \equiv m_i \pmod{ p^{k_i}} \quad (1 \leq i \leq N)
\end{equation}
satisfying the additional condition that $\Delta(\vx_1, \dots, \vx_M) \not \equiv 0 \bmod p$.
\end{definition}

\begin{lemma}\label{auxilliary congruence count lemma}  We have the upper bound
\begin{equation}\label{auxilliary congruence count}
|\mathcal{B}^{\vsigma}_{p}(\vm;\vxi)| \leq k_1 \dotsm k_N\ p^{2Mk - K}.
\end{equation}
\end{lemma}

In order to prove Lemma \ref{auxilliary congruence count lemma}, we record a simple generalisation of Lagrange's theorem on the number of roots of a non-zero polynomial over an arbitrary field, a result which proves useful elsewhere.  

\begin{lemma}\label{generalised lagrange}
Let $\F$ be a field and $P \in \F[X_1, \dots, X_m]$ a non-zero polynomial.  Let $a_i$ denote the highest exponent of $X_i$ occurring in $P$.  If $A \subset \F$ is finite then
$$
\hash \set{\vx \in A^m : P(\vx) = 0} \leq (a_1+ \dots + a_m) |A|^{m-1}.
$$
\end{lemma}
The proof is a simple induction on the number of variables $m$, which we leave as an exercise for the reader.

\begin{proof}[Proof of Lemma \ref{auxilliary congruence count lemma}]

Let $\mathcal{D}(\vn)$ denote the number of elements $(\vx_1, \dots, \vx_M)$ in the set $ \mathcal{B}^{\vsigma}_p(\vm; \vxi)$ satisfying the stronger congruence
$$
\sum_{j=1}^M  \sigma_j \vF(\vx_j- \vxi ) \equiv \vn \pmod{p^{k}}.
$$
Then
\begin{equation}\label{B to D1}
\begin{split}
|\mathcal{B}^{\vsigma}_p(\vm; \vxi)| & \leq \twosum{1 \leq n_1 \leq p^{k}}{n_1 \equiv m_1 \bmod p^{k_1}}\dots \twosum{1 \leq n_N \leq p^{k}}{n_N \equiv m_N \bmod p^{k_N}} \mathcal{D}(\vn)\\
& \leq p^{(kN- K)} \max_{\vn} \mathcal{D}(\vn).
\end{split}
\end{equation}

For each tuple $(x_1, x_2, \dots, x_{2M-1}, x_{2M})$ counted by $\mathcal{D}(\vn)$ there are at most $p^{(2M-N)k}$ choices for $x_{i}$ with $i > N$.  Fix such a choice, and define the polynomials
$$
f_i(x_{1}, \dots, x_{N}) = \sum_{j=1}^M \sigma_jF_i(x_{2j-1}- \xi_1, x_{2j}- \xi_2) - n_i\quad (1 \leq i \leq N).
$$
Then $f_1, \dots, f_N$ are polynomials in $\Z[x_1, \dots, x_N]$ with $\deg f_i = k_i$.  By Theorem 1 of Wooley \cite{wooley96}, the number of integer tuples $1 \leq (x_1, \dots, x_N) \leq p^k$ satisfying both
$$
f_i(x_1, \dots, x_N) \equiv 0 \mod p^k \qquad (1\leq i \leq N)
$$
and
$$
\det\Bigbrac{ \frac{\partial f_i}{\partial x_j} (\vx)}_{i, j} \not \equiv 0 \mod p
$$
is at most $(\deg f_1) \dotsm (\deg f_N) = k_1 \dotsm k_N$.  One can check using \eqref{matrix translation invariance}, that for $(u,v) = (1,0)$ or $(0,1)$, we have
$$
F_i^{u,v} (\vx_j - \vxi) = F_i^{u,v}(\vx_j) + \sum_{l > i} (\Xi_{-\vxi})_{il} F_l^{u,v}(\vx_j).
$$ Hence it follows that
$$
\Bigabs{\det\Bigbrac{ \frac{\partial f_i}{\partial x_j} (x_1, \dots, x_N)}_{i, j} }= \bigabs{\Delta(x_1, x_2, \dots, x_{2M-1}, x_{2M})}.
$$
So there are at most $k_1 \dotsm k_N$ choices for $(x_1, \dots, x_N)$ with $(x_1, \dots, x_{2M})$ counted by $\mathcal{D}(\vn)$.  Thus
\begin{equation}\label{final D2}
\mathcal{D}(\vn) \leq k_1\dotsm k_N p^{(2M-N)k}.
\end{equation}
Putting \eqref{B to D1} and \eqref{final D2} together, we obtain the lemma.
\end{proof}

Lemma \ref{auxilliary congruence count lemma} allows us to count non-singular solutions, which we have still to define.  The remaining singular solutions are counted by the following lemma.  First a definition.
\begin{definition}
Define $\mathcal{S}_{t}(X)$ to be the set of 
$$
(\vx_1, \dots, \vx_t) \in [X]^{2t}
$$
such that for any function $h : [M] \to [t]$ we have the identity
$$
\Delta(\vx_{h(1)},\dots, \vx_{h(M)}) = 0.
$$
\end{definition}

\begin{lemma}\label{highly singular lemma}
Let $\Phi$ be a non-degenerate binary form of degree $k$ and differential dimension $N$.  Setting $M = \ceil{N/2}$, we have the upper bound
\begin{equation}\label{highly singular bound}
|\mathcal{S}_{ t}(X)| \leq Mt^M (2k)^t X^{t+M -1}.
\end{equation}
\end{lemma}

\begin{proof}
The result follows trivially if $t < M$, so we may assume that $t \geq M$.  Let us define a sequence of non-zero polynomials $\Delta_i(\vx_1, \dots, \vx_i)$ for $i = 0,1, \dots, M$.  We begin by setting $\Delta_M = \Delta$.  Suppose we have constructed $\Delta_i$ with $i >1$.  Let us write $\vx^{\va}$ for the monomial $x_1^{a_1}
x_2^{a_2}$.  Of the monomials $\vx_1^{\va_1} \dotsm \vx_i^{\va_i}$ occurring in $\Delta_i$, let $\vb_i$ denote the maximum in the lexicographical ordering over all $\va_i$.  It follows that there exist polynomials $\Delta_{i-1}(\vx_1, \dots, \vx_{i-1})$ and $R_i(\vx_1, \dots, \vx_i)$ such that
\begin{equation}\label{recursive polynomials}
\Delta_i(\vx_1, \dots, \vx_i) = \Delta_{i-1}(\vx_1, \dots, \vx_{i-1})\vx_i^{\vb_i} + R_i(\vx_1, \dots, \vx_i).
\end{equation}
Moreover, we may assume $\Delta_{i-1}$ is non-zero and that every monomial $\vx_1^{\va_1}\dotsm\vx_i^{\va_i}$ occurring in $R_i$ satisfies $\va_i \prec \vb_i$, where $\prec$ denotes the (strict) lexicographical ordering.  For consistency, let us set $\Delta_0 = 1$ and $R_1 = 0$.  For each $i$ in the range $1\leq i \leq M$, define $\mathcal{T}_i$ to be the set of $(\vx_1, \dots, \vx_t) \in [X]^{2t}$ satisfying both of the following conditions:
\begin{enumerate}[(i)]
\item For any $h : [i] \to [t]$ we have
$$
\Delta_i(\vx_{h(1)}, \dots, \vx_{h(i)}) = 0,
$$

\item For each $j < i$ there exists $h : [j] \to [t]$ such that
$$
\Delta_j(\vx_{h(1)}, \dots, \vx_{h(j)}) \neq 0.
$$
\end{enumerate}
Then we have that 
$$
\mathcal{S}_{t}(X) \subset \bigcup_{1 \leq i \leq M} \mathcal{T}_i.
$$

Let $(\vx_1, \dots, \vx_M) \in \mathcal{T}_i$.  Then there exists some $h: [i-1] \to [t]$ such that 
$$
\Delta_{i-1}(\vx_{h(1)}, \dots, \vx_{h(i-1)}) \neq 0,
$$ 
yet for all $j \notin \set{h(1), \dots, h(i-1)}$, the identity \eqref{recursive polynomials} tells us that
\begin{equation}\label{recursive zeroes}
0  = \Delta_{i-1}(\vx_{h(1)}, \dots, \vx_{h(i-1)})\vx_j^{\vb_j} + R_j(\vx_{h(1)}, \dots, \vx_{h(i-1)}, \vx_j).
\end{equation}
Since the two-variable polynomial
$$
Q(\vX) = \Delta_{i-1}(\vx_{h(1)}, \dots, \vx_{h(i-1)})\vX^{\vb_j} + R_j(\vx_{h(1)}, \dots, \vx_{h(i-1)}, \vX)
$$
is non-zero, it follows from Lemma \ref{generalised lagrange} that for each $j$, the number of $\vx_j$ satisfying \eqref{recursive zeroes} is at most
$$
(b_{j1} + b_{j2})X \leq 2k X.
$$
There are trivially at most $X^{2(i-1)}$ choices for $(\vx_{h(1)}, \dots, \vx_{h(i-1)})$, and at most $t^{i-1}$ choices for $h : [i-1] \to [t]$.  Thus
$$
|\mathcal{T}_i| \leq t^{i-1} (2k)^{t-i+1} X^{ 2(i-1) + (t - i +1)} =  t^{M}(2k)^{t}  X^{ t + i -1}.
$$
Hence
$$
|\mathcal{S}_{ t}(X)| \leq Mt^{M}(2k)^t  X^{t+ M -1}.
$$
\end{proof}

We can now implement the results obtained  so far in this section to prove the following lemma, which encodes the basic iterative relation underlying our $p$-adic approach to bounding $J_{s, \Phi}$.  Again, we begin with a definition.

\begin{definition}
Given a prime number $p$, $\vxi \in \Z^2$ and $\vsigma \in \set{-1, 1}^M$, define the exponential sums 
\begin{align*}
\mathfrak{f}_p(\valpha; \vxi) & = \twosum{\vx \in [X]^2}{\vx \equiv \vxi \bmod p} e(\valpha \cdot \vF(\vx)),\\
\mathfrak{F}_p^{\vsigma}(\valpha) & = \twosum{(\vx_1, \dots, \vx_M) \in [X]^{2M}}{\Delta(\vx_1, \dots, \vx_M) \not \equiv 0 \bmod p} e\Bigbrac{ \valpha \cdot \sum_{j =1}^M \sigma_j \vF(\vx_j)}.
\end{align*}
\end{definition}

 \begin{lemma}\label{iteration initiation}
Let $s \geq M$.  Then there exists $\vxi \in \Z^2$, $\vsigma \in \set{-1, 1}^M$ and a prime $p$ in the range $ X^{1/k} < p \leq 2X^{1/k}$ such that 
\begin{equation}\label{basic iterative relation equation}
J_{s,\Phi}(X) \ll X^{2s + M -1} + p^{4(s-M)} \oint |\mathfrak{F}_p^{\vsigma}(\valpha)|^2 |\mathfrak{f}_p(\valpha; \vxi)|^{2(s-M)}\intd\valpha.
\end{equation}
\end{lemma}

\begin{proof}
If $(\vx_1, \dots, \vx_M) \in [X]^{2M}$, then each $ij$-entry of the matrix $\Jac(\vx_1, \dots, \vx_M)$, defined in \eqref{Jac defn}, is of order $O(X^{k_i -1})$.  Hence there exists a constant $C = C(\Phi)$ such that for any $(\vx_1, \dots, \vx_N) \in [X]^{2M}$ we have
$$
|\Delta(\vx_1, \dots, \vx_M)| \leq CX^{K-N}.
$$
By the prime number theorem, for all $X \gg_{\Phi} 1$ we have
\begin{equation}\label{PNT bound}
\pi(2X^{1/k}) - \pi(X^{1/k}) \geq \frac{k\log C}{\log X} + k(K-N).
\end{equation}
Let $T$ be the smallest positive integer bounded below by the right-hand side of \eqref{PNT bound}, and let $\mathcal{P}$ denote the set of the $T$ smallest primes in the interval $X^{1/k} < p \leq 2X^{1/k}$.  Then
$$
\prod_{p \in \mathcal{P}} p > X^{T/k} \geq C X^{K-N}.
$$
In particular, for each $(\vx_1, \dots, \vx_M) \in [X]^{2M}$ with $\Delta(\vx_1, \dots, \vx_M) \neq 0$, there must exist $p \in \mathcal{P}$ such that 
$$
\Delta(\vx_1, \dots, \vx_M) \not \equiv 0 \pmod p.
$$

Now $J_{s, \Phi}(X)$ counts tuples $(\vx_1, \dots, \vx_{2s}) \in [X]^{4s}$ which satisfy 
\begin{equation}\label{vF version}
\sum_{i=1}^{s} (\vF(\vx_{2i-1}) - \vF(\vx_{2i})) = 0.
\end{equation}
Let $\mathcal{T}$ denote the number of such tuples which are not contained in $\mathcal{S}_{2s}(X)$.  Then by Lemma \ref{highly singular lemma} we have
$$
J_{s,\Phi}(X) \ll_{s, \Phi} X^{2s + M-1} + \mathcal{T}.
$$
If $(\vx_1, \dots, \vx_{2s})$ denotes a tuple counted by $\mathcal{T}$, then it satisfies \eqref{vF version} and there exists $h : [M] \to [2s]$ such that $\Delta(\vx_{h(1)}, \dots, \vx_{h(M)}) \neq 0$.  Hence for each such choice of tuple and function $h$, there exists a prime $p \in \mathcal{P}$ such that
$$
\Delta(\vx_{h(1)}, \dots, \vx_{h(M)}) \not \equiv 0 \pmod p.
$$
Notice that by the definition of the determinant $\Delta$, such a choice of $h$ must be injective.
Since there are $O(1)$ choices for $h$, and $O(1)$ choices for a prime $p \in \mathcal{P}$, we see that there exists $p$ and $\vsigma \in \set{-1, 1}^M$ such that
\begin{align*}
\mathcal{T} & \ll \oint |\mathfrak{F}_p^{\vsigma}(\valpha)| |f(\valpha)|^{2s -M} \intd \valpha\\
	& \leq \Bigbrac{\oint |\mathfrak{F}_p^{\vsigma}(\valpha)|^2 |f(\valpha)|^{2(s-M)} \intd \valpha}^{1/2} J_{s,\Phi}(X)^{1/2}.
\end{align*}
Thus
\begin{equation}\label{final initiation bound}
J_{s,\Phi}(X)  \ll X^{2s + M -1} + \oint |\mathfrak{F}_p^{\vsigma}(\valpha)|^2 |f(\valpha)|^{2(s-M)} \intd \valpha.
		\end{equation}
By the triangle inequality
$$
|f(\valpha)|^{2(s-M)} =\Bigabs{ \sum_{ \vxi \in [p]^2}\mathfrak{f}_p (\valpha; \vxi)}^{2(s-M)} \leq p^{4(s-M)} \max_{\vxi} |\mathfrak{f}_p (\valpha; \vxi)|^{2(s-M)}.
$$
Incorporating this into \eqref{final initiation bound}, we obtain the lemma.
\end{proof}

The following lemma eventually allows us to conclude that the first term on the right-hand side of \eqref{basic iterative relation equation} is smaller than our hoped for upper bound.

\begin{lemma}\label{Delta inequality corollary}  Let $\Phi$ be a non-degenerate binary form of degree $k$, differential dimension $N$ and differential degree $K$.  Set $M = \ceil{N/2}$.  Then we have the inequality
 \begin{equation}\label{Delta inequality}
 K/k\leq M+\trecip{2}.     
 \end{equation}
 \end{lemma}
\begin{proof} We  establish this result in a series of claims.  Let $N_l$ denote the size of the set $\set{i : k_i = l}$.

\textbf{Claim 1.}  \emph{Let $2 \leq l \leq k$.  Then $N_{l-1} \geq N_{l}$ or $N_l = l+1$. }

Let $G_1, \dots, G_{m}$ denote a basis of forms for the space $\spn\set{F_i : k_i = l}$, so that $m = N_l$.  Let $d_j$ denote the degree of the one variable polynomial $G_j(x,1)$.  By performing a linear transformation we may assume that $d_1 > d_2 > \dots > d_m$.

Suppose there exists $i$ for which $d_i < l+1 - i$.  Let $i$ denote the minimal such index.  Then each of the one variable polynomials $G_j^{1,0}(x,1)$ with $1 \leq j < i$ has degree $d_j - 1$, whilst each of the polynomials $G_j^{0,1}(x,1)$ with $j\geq i$ has degree $d_j$.  Since 
\begin{align*} d_1 -1 > d_2 -1 > \dots > d_{i-1} - 1 & = l +1 - i \\ & > d_i > \dots > d_m,\end{align*}
the polynomials $G_1^{1,0}, \dots, G_{i-1}^{1,0}, G_i^{0,1},\dots, G_m^{0,1}$ are a linearly independent subset of the space $\spn\set{F_i : k_i = l-1}$.  It follows that $N_{l-1} \geq N_l$.

Next suppose that for all $i$ we have $d_i = l+i -1$.  If $N_l < l+1$ then, as above, $G_1^{1,0}, \dots, G_m^{1,0}$ form a linearly independent subset of $\spn\set{F_i : k_i = l-1}$ of size $N_l$ and we are done.  The only remaining possibility is that $N_l = l+1$.  This establishes Claim 1.

\textbf{Claim 2.} \emph{For $2 \leq l \leq k$ we have the inequality $$N_{l-1} \leq N_l +1.$$}
  
Let $G_1, \dots, G_{m}$ denote a basis for $\spn\set{F_i : k_i = l}$.  Then for each $u,v$ with $l = k - u- v$, the form $\Phi^{u,v}$ is a linear combination of $G_1, \dots, G_m$.  It follows that each $\Phi^{u+1, v}$ is a linear combination of $G_1^{1,0}, \dots, G_{m}^{1, 0}$.  If $u>0$ then $\Phi^{u, v+1}$ is also a linear combination of the $G_i^{1,0}$, since $(u, v+1) = (u'+1, v')$ for some $u', v' \geq 0$ and $l = d - u' -v'$.  Thus
$$
\spn\set{F_i : k_i = l-1} = \spn\set{G_1^{1, 0}, \dots, G_{m}^{1, 0},\ \Phi^{0, d-l+1}}.
$$
Clearly this latter space has dimension at most $N_l +1$, which is what we require.

\textbf{Claim 3.}  \emph{If $1 \leq l \leq k/2$ then $N_{k-l} \leq N_l$.}

Let $1 \leq l < k/2$.  If $N_i < i+1$ for all $l < i \leq k-l$ then by Claim 1 we have 
$$
N_{k-l} \leq \dots \leq N_{l+1} \leq N_l .
$$
Next suppose $N_i = i+1$ for some $l < i \leq k-l$.  Since $\spn\set{F_j : k_j = i}$ is a subspace of the $(i+1)$-dimensional space
$$
\spn\set{ x^{i-j} y^{j} : 0 \leq j \leq i},
$$
these spaces must in fact coincide.  Taking derivatives, we see that $\spn\set{F_j : k_j = l}$ coincides with $\spn\set{x^{l-j}y^j : 0 \leq j \leq l}$, so $N_l = l+1$. By Claim 2, we have 
$$N_{k - l} \leq N_{k - (l-1)} +1 \leq \dots \leq N_{k} + l \leq 1 +l = N_{l},$$
the last inequality being a consequence of $N_k =1$.  This establishes Claim 3.

Finally, we use Claim 3 to prove the lemma.  Observing that $\trecip{2}+ M \geq (N +1)/2$, the inequality \eqref{Delta inequality} therefore reduces to showing that 
$$
 \sum_{l=1}^{k} l N_l \leq  \tfrac{k}{2} + \sum_{l=1}^{k} \tfrac{k}{2} N_l.
$$
Re-arranging, we need only show
$$
\sum_{{k}/2 < l \leq {k}} \brac{l - \tfrac{{k}}{2}} N_l \leq \tfrac{k}{2}+\sum_{1 \leq l < {k}/2} \brac{\tfrac{{k}}{2} -l} N_l.
$$
Changing variables from $l$ to ${k}-l$ on the left hand side leaves us the task of proving
$$
\tfrac{{k}}{2}N_{{k}} +\sum_{1 \leq l < {k}/2} \brac{\tfrac{{k}}{2} -l } N_{{k}-l} \leq \tfrac{k}{2}+ \sum_{1 \leq l < {k}/2} \brac{\tfrac{{k}}{2} - l} N_l.
$$
This last inequality follows from Lemma \ref{no one dim space} and the fact that $N_{{k}} =1$.  
\end{proof}
%
 
 The final lemma proved before we deduce Theorem \ref{intro VMVT} bounds the second term on the right-hand side of \eqref{basic iterative relation equation}.
 
 \begin{lemma}\label{non-singular iteration lemma}
Suppose that $s \geq M$ and $X^{1/k} \leq p \leq X$.  Then
\begin{equation}\label{non-singular iteration equation}
 \oint |\mathfrak{F}_p^{\vsigma}(\valpha)|^2 |\mathfrak{f}_p(\valpha; \vxi)|^{2(s-M)}\intd\valpha \ll X^{2M} p^{2Mk - K} J_{s-M, \Phi}(2X/p).
\end{equation}
\end{lemma}
 
 \begin{proof}
The left-hand side of \eqref{non-singular iteration equation} counts tuples $(\vx_1, \vy_1, \dots , \vx_s, \vy_s) \in [X]^{4s}$ satisfying the Diophantine equations
\begin{equation}
\sum_{j=1}^M \sigma_j\bigbrac{\vF(\vx_j) - \vF(\vy_j)} = \sum_{j=M+1}^{s} \bigbrac{\vF(\vx_j) - \vF(\vy_j)},
\end{equation}
under the additional constraints that both $\Delta(\vx_1, \dots,\vx_M)$ and $\Delta(\vy_1, \dots, \vy_M)$ are non-zero modulo $p$, and for all $j > M$ we have $\vx_j \equiv \vy_j \equiv \vxi \pmod p$.
Translation-invariance \eqref{equation translation invariance} and homogeneity together imply that
$$
\sum_{j=1}^M \sigma_j \Bigbrac{F_i(\vx_j - \vxi) - F_i(\vy_j - \vxi)} \equiv 0 \mod p^{k_i} \qquad (1\leq i \leq N).
$$
Fix a choice of $(\vy_1, \dots, \vy_M) \in [X]^{2M}$ and set
$$
\vm = \sum_{j=1}^M \sigma_j F_i(\vy_j - \vxi).
$$
Let $\overline{x}$ denote the residue class of $x \in \Z$ modulo $p^k$.  Then $(\overline{\vx}_1, \dots, \overline{\vx}_M) \in \mathcal{B}_{p}^{\vsigma}(\vm; \vxi)$.  Since $p^k \geq X$, the map 
$$
(\vx_1, \dots, \vx_M) \mapsto (\overline{\vx}_1, \dots, \overline{\vx}_M)
$$
is injective when restricted to $[X]^{2M}$.  Hence, there are at most $|\mathcal{B}^{\vsigma}_p(\vm; \vxi)|$ choices for $(\vx_1, \dots, \vx_M)$.  Set 
$$
\vn = \sum_{j=1}^M \sigma_j \Bigbrac{F_i(\vx_j - \vxi) - F_i(\vy_j - \vxi)}.
$$
Then for each fixed choice of $(\vx_1, \vy_1, \dots, \vx_M, \vy_M)$, the number of choices for the remaining $\vx_{j}, \vy_{j}$ $(j > M)$ is at most
$$
\oint |\mathfrak{f}_p(\valpha ; \vxi)|^{2(s-M)}e(- \valpha \cdot \vn) \intd \valpha \leq \oint |\mathfrak{f}_p(\valpha ; \vxi)|^{2(s-M)} \intd \valpha.
$$

Employing Lemma \ref{auxilliary congruence count lemma}, it remains to establish that
\begin{equation}\label{crude bounds}
\oint |\mathfrak{f}_p(\valpha; \vxi)|^{2t} \intd \valpha \ll_t J_t(2X/p).
\end{equation}
We may assume $\vxi \in [p]^2$.  The integral in \eqref{crude bounds} then counts the number of solutions to the system
$$
\sum_{i=1}^{t} (\vF(\vxi + p\vx_i) - \vF(\vxi + p\vy_i)) = 0,\quad 
\vx_i, \vy_i  \in \sqbrac{0, \tfrac{X- \xi_1}{p}} \times    \sqbrac{0, \tfrac{X- \xi_2}{p}}.
$$
By \eqref{equation translation invariance} and homogeneity, this equals the number of solutions of the system
$$
\sum_{i=1}^t (\vF(\vx_i) - \vF(\vy_i)) = 0,\quad
\vx_i, \vy_i  \in  \sqbrac{1, \tfrac{X- \xi_1}{p}+1} \times   \sqbrac{1, \tfrac{X- \xi_2}{p}+1}.
$$
Since we may assume $X \geq p$, the result follows.\end{proof}
 
 To conclude this section, we prove our mean value theorem.
 
 \begin{proof}[Proof of the Theorem \ref{intro VMVT}]
 We proceed by induction on $\floor{s/M}\geq 0$.  The basis case is equivalent to $J_{s,\Phi}(X) \ll X^{4s}$, which is trivial.  

Let us suppose that $\floor{s/M} \geq 1$.  Combining Lemma \ref{iteration initiation} and Lemma \ref{non-singular iteration lemma}, we see that there exists a prime $p$ in the interval $(X^{1/k}, 2X^{1/k}]$ such that
\begin{equation}\label{key iteration inequality}
J_{s,\Phi}(X) \ll X^{2s + M -1} + p^{4(s-M) - K} X^{4M} J_{s-M, \Phi}(2X/p).
\end{equation}
Combining this with our induction hypothesis implies that
$$
J_{s,\Phi}(X)  \ll X^{2s + M -1} +  X^{4s- K + \Delta_{s} }.
$$
It therefore remains to show that $2s + M -1 \leq 4s - K + \Delta_s$, which we also prove by induction on $\floor{s/M} \geq 1$.  The basis case follows directly from the estimate $K/k \leq M+\recip{2}$ of Lemma \ref{Delta inequality corollary}.  Let us suppose $\floor{s/M}\geq 2$, then by induction $2s + M -1$, being equal to $2M + 2(s-M) + M -1$, is at most
\begin{align*}
2M + 4(s-M) - K + K(1-\trecip{k})^{\floor{\frac{s}{M}}-1} & = 4s - K + \Delta_s + \tfrac{K}{k}(1-\trecip{k})^{\floor{\frac{s}{M}}-1} - 2M\\
			& \leq 4s - K + \Delta_s + \tfrac{K}{k} - 2M.
\end{align*}
Since $K = \sum_{i=1}^N k_i \leq N k$ and $N \leq 2M$, we have $\tfrac{K}{k} - 2M \leq 0$, which completes the proof.\end{proof}
 
 \section{Weyl-type estimates}\label{Weyl-type section}
 
 \begin{definition}  Let us say $\Delta = \Delta_s$ is an \emph{admissible exponent} for $(s, \Phi)$ if there exists a constant $C = C(s, \Phi)$ such that for any $X \geq 1$ we have the bound $J_{s, \Phi}(X) \leq CX^{4s - K + \Delta}$.
 \end{definition}
 
 The aim of this section is to prove the following Weyl-type estimate.
 
 \begin{theorem}\label{minor arc estimate theorem}
 Let $\Delta$ be an admissible exponent for $(s, \Phi)$ and let $\sigma <\tfrac{1-3\Delta}{6s + 3}$.  Then for any $\eps > 0$ there exists a constant $C = C(s, \Phi, \eps)$ such that if $X \geq C$ and
 \begin{equation}\label{f lower bound}
|f(\valpha; X)| \geq X^{2 - \sigma},
\end{equation}
then there exists integers $q, a_1, \dots, a_N$ such that $1 \leq q \leq X^{k\sigma + \eps}$ and $|q \alpha_j - a_j| \leq X^{k\sigma +\eps - k_j}$ for all $1 \leq j \leq N$.
 \end{theorem}
 
The proof of Theorem \ref{minor arc estimate theorem} uses Vinogradov's method, a general heuristic for which is neatly described in \cite[\S 8.5]{iwanieckowalski}.  We model our argument on a version of the method due to Vaughan \cite[Chapter 5]{vaughan97}, with a later improvement due to Baker \cite[Chapter 4]{baker86}.

 Let $\gamma_{ij}(\vxi)$ denote the $ij$-entry of the matrix $\Xi_{\vxi}$ occurring in \eqref{matrix translation invariance}.  Then for any $\vy$ we have the identity
$$
F_i(\vx + \vy) - F_i(\vy) = \sum_{j} \gamma_{ij}(\vy) F_j(\vx) = F_i(\vx) + \sum_{j > i} \gamma_{ij}(\vy) F_j(\vx).
$$
Set 
$$
\gamma_j(\vy) = \gamma_j(\vy; \valpha) =  \sum_{i < j}  \alpha_i\gamma_{ij}(\vy),
$$
then for all $\vx$ we have
\begin{equation}\label{shift identity}
\valpha \cdot \Bigbrac{\vF(\vx + \vy ) -  \vF(\vx)- \vF(\vy)} = \sum_{j = 2}^N \gamma_j(\vy) F_j(\vx).
\end{equation}

In the following result, and the remainder of the paper, we use $\norm{\beta}$ to denote the smallest distance from $\beta$ to an integer. 
 
 \begin{lemma}\label{Vinogradov method lemma}
Let $\Delta$ denote an admissible exponent for $(s, \Phi)$ and let $\mathcal{S}$ be a subset of $[X]^2$ of size $S$ such that for any distinct $\vy, \vz \in \mathcal{S}$ there exists $2 \leq j \leq N$ with
\begin{equation}\label{spacing condition}
\norm{ \gamma_j(\vy) - \gamma_j(\vz)} > X^{-k_j}.
\end{equation}
 Then one has
\begin{equation}\label{Vaughan Weyl-type estimate}
|f(\valpha; X)| \ll X^{2}\log (2X)^2\brac{X^\Delta/S}^{1/(2s)} .
\end{equation}
\end{lemma}

\begin{proof}

Averaging, we see that $f(\valpha)$ is equal to
\begin{equation}
\recip{S} \sum_{\vy \in \mathcal{S}}\sum_{\vx \in (-X, X)^2 } e\bigbrac{ \valpha\cdot \vF(\vx + \vy)}1_{[X]^2 - \vy}(\vx).
\end{equation}
By orthogonality
$$
1_{[X]^2 - \vy}(\vx) = \oint e(\vbeta\cdot \vx)\sum_{\vn \in [X]^2 - \vy}  e(- \vbeta\cdot\vn) \intd\vbeta.
$$
Since $\sum_{1-y \leq n \leq X - y}  e(- \beta n) \leq \min\set{X, \norm{\beta}^{-1}}$, the sum $f(\valpha)$ is at most
\begin{align*}
\recip{S}  \oint&\Bigabs{\twosum{\vy \in \mathcal{S}}{\vx \in (-X, X)^2 }  e\bigbrac{ \valpha\cdot \vF(\vx + \vy)+\vbeta\cdot \vx}}\min\set{X, \norm{\beta_1}^{-1}}\min\set{X, \norm{\beta_2}^{-1}} \intd\vbeta\\
& \ll \frac{\log(2X)^2}{S}\sup_{\vbeta}\sum_{\vy \in \mathcal{S}}\abs{g(\vy, \vbeta)},
\end{align*}
where
$$
g(\vy, \vbeta)= \sum_{\vx \in (-X, X)^2 }    e\brac{ \valpha\cdot \vF(\vx + \vy)+\vbeta\cdot \vx }.
$$
It follows from H\"{o}lder's inequality that there exists $\vbeta \in \T^2$ such that 
\begin{equation}\label{shifted average estimate}
|f(\valpha)|^{2s} \ll \frac{\log(2X)^{4s}}{S}\sum_{\vy \in \mathcal{S}}\abs{g(\vy, \vbeta)}^{2s}.
\end{equation}

Let $C = C(s, \Phi)$ be any constant such that for all $1 \leq i \leq N$ we have 
$$
\Bigabs{ \sum_{j=1}^s F_i(\vx_j)} \leq CX^{k_i} \quad (\vx_j \in (-X, X)^2).
$$
Define $\vF'(\vx) = (F_i(\vx))_{2\leq i \leq N}$, $\mathcal{N} = \prod_{2 \leq i \leq N} [-CX^{k_i}, CX^{k_i}]$ and 
$$
a(\vn) = \twosum{\vx_1, \dots, \vx_s \in (-X, X)^2}{\vn = \sum_{j=1}^s \vF'(\vx_j)} e\Bigbrac{ \valpha\cdot \sum_{j=1}^s \vF(\vx_j) + \vbeta \cdot \sum_{j=1}^{s} \vx_j}.
$$
Then by \eqref{shift identity} we have
\begin{align*}
\Bigabs{\sum_{\vx \in (-X, X)^2} &e(\valpha\cdot \vF(\vx+\vy) + \vbeta\cdot\vx)}^{2s}\\ & = \Bigabs{\sum_{\vx_1, \dots, \vx_s \in (-X, X)^2} e\Bigbrac{\valpha\cdot \sum_{j=1}^s \vF(\vx_j) + \vbeta \cdot \sum_{j=1}^s \vx_j}e\Bigbrac{\vgamma(\vy) \cdot \sum_{j=1}^s \vF'(\vx_j)}}^{2}\\
& = \Bigabs{\sum_{\vn \in \mathcal{N}} a(\vn) e(\vgamma(\vy)\cdot \vn)}^2.
\end{align*}
By \eqref{spacing condition} and the multi-dimensional version of the large sieve inequality (see for example Vaughan \cite[Lemma 5.3]{vaughan97}), we have
\begin{equation}\label{large sieve application}
\sum_{\vy \in \mathcal{S}}\Bigabs{\sum_{\vn \in \mathcal{N}} a(\vn) e(\vgamma(\vy)\cdot \vn)}^2   \ll_{s, \Phi} \sum_{\vn \in \mathcal{N}} |a(\vn)|^2 \prod_{2 \leq i \leq N} X^{k_i}.
\end{equation}
Let $\ve_1$ denote the first standard basis vector.  Recalling that $J_{s, \Phi}(X; \vm)$ denotes the number of $(\vx, \vy) \in [X]^{4s}$ satisfying
 $$
 \sum_{i=1}^s \bigbrac{\vF(\vx_i) - \vF(\vy_i)} = \vm,
 $$
 we have, by translation invariance, that
\begin{equation}\label{L 2 a estimate}
\sum_{\mathcal{N}} |a(\vn)|^2 \leq \sum_m J_{s, \Phi}(2X + 1; m\ve_1) ,
\end{equation}
where the summation over $m$ ranges over a set of size $CX^{k_1}$.
If $\Delta$ is an admissible exponent for $(s,\Phi)$ then we deduce that the right-hand side of \eqref{L 2 a estimate} is of order $O(X^{4s - (K-k_1) + \Delta})$.  Putting these facts together with \eqref{large sieve application}, we obtain
$$
|f(\valpha)|^{2s} \ll_{s, \Phi} \frac{\log(2X)^{4s}}{S} X^{4s + \Delta}.
$$
\end{proof}

 \begin{remark}
 In the proof of Lemma \ref{Vinogradov method lemma}, we used $J_{s, \Phi}(X)$ to bound the number $J'_{s, \Phi}(X)$ of solutions $(\vx, \vy) \in [X]^{4s}$ to the smaller system of equations
 $$
 \sum_{i=1}^s \bigbrac{\Phi^{u,v}(\vx_i) - \Phi^{u,v}(\vy_i)} = 0 \qquad (u+v \geq 1).
 $$
 We note that the methods of \S \ref{mean value section} translate almost verbatim to yield the bound $J'_{s, \Phi}(X) \ll X^{4s - (K-k) + \Delta}$, where for $s = l\ceil{(N-1)/2}$ we have 
  $$
 \Delta \leq (K-k)(1 - \trecip{k-1})^{l}.
 $$
 This $\Delta$ is clearly superior to that obtained  using $J_{s, \Phi}(X)$.  However, at the level of detail we are concerned with, this makes little difference to our final results, and increases the expositional complexity of \S \ref{mean value section} considerably.
 \end{remark}

In order to obtain a set $\mathcal{S}$ satisfying the spacing condition \eqref{spacing condition}, we relate this condition to the Diophantine approximation of our original coefficients $\alpha_i$.  This is the content of the following lemma.  Unfortunately, the most direct approach allows us to control the spacing of the $\gamma_j(\vy)$ only according to the Diophantine approximation of a proper subset of the $\alpha_i$.  We first define this subset.

\begin{definition}\label{Im def}
We assume throughout that $I_1, I_2$ denote sets of indices whose union equals $\set{ r \in [N]: k_r \geq 2}$ and such that if $\set{i, j} = \set{1, 2}$ then both of the following conditions hold
\begin{equation}\spn\set{F_r(x_1,x_2) : r \in I_i} \cap \Q[x_j] = \set{0},\end{equation}
\begin{equation}\set{F_r(x_1,x_2) : r \notin I_i} \subset \Q[x_j].\end{equation}
Making a linear transformation of the $F_i$ if necessary, we can always guarantee the existence of such $I_1$ and $I_2$.
\end{definition}

\begin{lemma}\label{one coeff spacing lemma}
Let $m \in\set{1,2} $.  There exists an absolute constant $C = C(\Phi)$ and a positive integer $L \leq C$ such that for any $y, z \in [X]$ and $\valpha \in \T^N$, if
\begin{equation}\label{gamma spacing condition}
\norm{\gamma_j(y\ve_m) - \gamma_j(z\ve_m)} \leq X^{-k_j}\quad  \text{for all } 2\leq j \leq N,
\end{equation}
then
\begin{equation}\label{alpha spacing condition}
\norm{L\alpha_i(y - z)} \leq CX^{1-k_i}\quad \text{for all $i \in I_m$}.
\end{equation}
\end{lemma}

\begin{proof}
Let us suppose that $m=1$, the case $m=2$ being similar.  By Taylor's formula, we have
\begin{equation}\label{taylor form}
F_i(\vx + y\ve_1) = F_i(\vx) + \sum_{r=1}^{k_i-1} \frac{y^r}{r!} F_i^{r,0}(\vx) + F_i(y\ve_1).
\end{equation}
Since $F_1, \dots, F_N$ are a spanning subset of $\set{ \Phi^{u,v} : 0 \leq u+v < k}$, there must exist rationals $\lambda_{ij}^{r,0}$ such that for $1 \leq r < k_i$ we have 
\begin{equation}\label{Fs span derivatives}
F_i^{r, 0} = \twosum{j}{ k_j = k_i -r} \lambda_{ij}^{r,0} F_j.
\end{equation}
Combining \eqref{taylor form}, \eqref{Fs span derivatives} and \eqref{shift identity}, we obtain
\begin{align*}
\gamma_j(y\ve_1) = \gamma_j(y\ve_1; \valpha) &  = \sum_{r=1}^{k-k_j} \frac{y^r}{r!} \Biggbrac{\ \twosum{i}{ k_i = k_j + r} \alpha_i \lambda_{ij}^{r,0}}\\
									& = \sum_{r=1}^{k-k_j} \frac{y^r}{r!}\gamma_j^{(r)}, \ \text{say}.	
\end{align*} 
Notice that for $r\geq 2$ we have
\begin{align*}
\sum_j \lambda_{ij}^{r,0} F_j = F_i^{r,0}  & = \sum_t \lambda_{it}^{1,0} F_t^{r-1, 0}
						 = \sum_{t,j} \lambda_{it}^{1, 0} \lambda_{tj}^{r-1, 0} F_j.
\end{align*}
Hence, by linear independence, for all $i, j$ and $r \geq 2$ we have $\lambda_{ij}^{r,0} = \sum_t \lambda_{it}^{r-1, 0} \lambda_{tj}^{1,0}$.  It follows that 
\begin{equation}
\gamma_j(y\ve_1) = y\gamma^{(1)}_j + \sum_{r\geq 2} \frac{y^r}{r!} \twosum{t }{ k_t = k_j + r -1} \lambda_{tj}^{r-1, 0} \Biggbrac{\ \twosum{i}{ k_i = k_t +1} \alpha_i \lambda_{it}^{1,0}},
\end{equation}
and so $\gamma_j(y\ve_1) - \gamma_j(z \ve_1)$ must equal
\begin{equation}\label{gamma to gamma 1}
(y-z)  \gamma^{(1)}_j+ \sum_{2 \leq r \leq k - k_j} \frac{(y^{r-1}  + \dots + z^{r-1})}{r!} \twosum{t}{k_t = k_j + r -1} \lambda_{tj}^{r-1, 0}(y-z)\gamma_t^{(1)}.
\end{equation}

Let $L_1 = L_1(\Phi)$ be a positive integer such that $L_1\lambda_{ij}^{r,0}$ is an integer for all $i, j$ and $r$.  It follows almost immediately from the identity \eqref{gamma to gamma 1} and induction on the difference $k- k_j$, that there exists a constant $C_1 = C_1(\Phi)$ such that if  \eqref{gamma spacing condition} holds with $m=1$, then for any $2 \leq j \leq N$ we have
\begin{equation}\label{gamma 1 spacing condition}
\norm{L_1k! \gamma_j^{(1)}(y-z)} \leq C_1 X^{ - k_j}.
\end{equation}
Define the linear map
$$
A_d : (\beta_i)_{\substack{i \in I_1\\ k_i = d}} \mapsto \Bigbrac{\ \twosum{i \in I_1}{k_i = d} \beta_i \lambda_{ij}^{1,0}}_{\substack{ 2 \leq j \leq N \\ k_j = d-1}},
$$
so that 
$$
\brac{\gamma_j^{(1)}}_{k_j = d-1} = (\alpha_i)_{\substack{i \in I_1\\ k_i = d}} \cdot A_d.								$$
We claim that each $A_d$ is non-singular.  To this end suppose that $\vbeta A_d = 0$.  Then a little manipulation shows that
$$
\twosum{i \in I_1}{k_i = d} \beta_i F_i^{1, 0} = \twosum{j}{ k_j = d-1}\Bigbrac{ \twosum{i \in I_1}{k_i = d} \beta_i \lambda_{ij}^{1, 0}} F_j = 0.
$$
It follows that 
$$
\twosum{i\in I_1}{k_i = d} \beta_i F_i(x_1, x_2) \in \Q[x_2].
$$
This contradicts Definition \ref{Im def}, unless $\vbeta = 0$.  Hence for each $d$ there exists a rational matrix $B_d$ such that $A_dB_d = (I\ |\ 0)$, where $I$ denotes the identity matrix.  We therefore have that
\begin{equation}\label{gamma inversion}
\begin{split}
\brac{\gamma_j^{(1)}}_{k_j = d-1} \cdot B_d & =  (\alpha_i)_{\substack{i \in I_1\\ k_i = d}} \cdot (I\ |\ 0).
\end{split}
\end{equation}
Let $L_2 \in \N$ be such that all the matrices $L_2B_d$ ($2 \leq k \leq k$) have only integer entries.  Then by \eqref{gamma inversion} and \eqref{gamma 1 spacing condition}, for all $i\in I_1$ we have
$$
\norm{L_2 L_1 k! \alpha_i(y-z)} \ll_\Phi X^{1-k_i}.
$$
Taking $L = L_2 L_1 k!$, we obtain the lemma.\end{proof}

The next lemma combines Lemma \ref{Vinogradov method lemma} and Lemma \ref{one coeff spacing lemma}, a combination we record since we use it repeatedly in the proof of Theorem \ref{minor arc estimate theorem}.

\begin{lemma}\label{approximation to spacing lemma}
Fix $m \in \set{1, 2}$ and let $C$ and $L$ be as in Lemma \ref{one coeff spacing lemma}.  Suppose there exists a real $D \geq 1$ such that for any $y \in [X]$, there are at most $D$ elements $z \in [X]$ satisfying
\begin{equation}\label{single coefficient spacing bound}
\norm{L\alpha_i(y-z)} \leq CX^{1-k_i} \quad \text{for all $i \in I_m$}.
\end{equation}
Then we have
\begin{equation}
|f(\valpha;X)| \ll X^2 \log(2X)^2 \bigbrac{X^\Delta D/X}^{1/(2s)}.
\end{equation}
\end{lemma}

\begin{proof}
By Lemma \ref{Vinogradov method lemma}, it remains to prove that our assumptions imply the existence of a set $\mathcal{S} \subset [X]$ of size $|\mathcal{S}| \gg XD^{-1}$ such that for any $y, z \in \mathcal{S}$ with $y \neq z$ there exists $2 \leq j \leq N$ with 
\begin{equation}\label{good spacing again}
\norm{\gamma_j(y \ve_m) - \gamma_j(z \ve_m)} > X^{-k_j}.
\end{equation}
By Lemma \ref{one coeff spacing lemma}, the spacing \eqref{good spacing again} holds for some $j$ if there exists $i \in I_m$ such that $\norm{L\alpha_i(y-z)} > CX^{1 - k_i}$.  Define $G$ to be the graph on vertex set $[X]$ with $y$ adjacent to $z$ if and only if $\norm{L\alpha_i(y-z)} \leq CX^{1-k_i}$ for all $i \in I_m$.  This graph has maximal degree at most $D-1$, so (by the greedy algorithm) contains an independent set of vertices $\mathcal{S}$ of size at least $\floor{X}/D$ (as required).
\end{proof}

In order to use Lemma \ref{approximation to spacing lemma} to relate the size of $f(\valpha)$ to the simultaneous Diophantine approximation of \emph{all} the $\alpha_i$, including those with $k_i =1$, we must utilise major arc information.  This necessitates the discussion of the standard major arc auxiliary approximation to $f(\valpha)$.  

\begin{definition}   Define 
$$
S(q, \va) = \sum_{\vz \in [q]^2}e\brac{q^{-1}\va \cdot \vF(\vz)},\quad I(\vbeta; X) = \int_{[0,X]^2}e\brac{\vbeta\cdot \vF(\vgamma)} \intd\vgamma
$$
and
$$
V(\valpha; q, \va) = q^{-2} S(q, \va) I(\valpha - \va/q; X).
$$
\end{definition}

The following three results, which bound $S(q, \va)$, $I(\vbeta; X)$ and the difference $f(\valpha) - V(\valpha; q, a)$, prove useful both in this section and the next.

  \begin{lemma}\label{complete sum bound}
  Let $q \in \N$ and $\va \in \Z^N$.  Then for any $\eps > 0$ we have
  \begin{equation}\label{complete sum bound equation}
  S(q, \va) \ll_\eps (q, \va)^{2} q^{2 - \recip{k} + \eps}.
  \end{equation}
  \end{lemma}
  
\begin{proof}
Letting $q' = (q, \va)^{-1}q$ and $\va' = (q, \va)^{-1} \va$, we have $S(q, \va) = (q, \va)^2 S(q' , \va')$.  Hence it suffices to assume that $(q, \va) =1$.  Sorting the expression $\va \cdot \vF$ into monomials and using the linear independence of the forms $F_i$, we see that there exists an integer matrix $B$ with full row-rank such that 
$$
\va \cdot \vF(\vx) = \sum_{0 < i_1 + i_2 \leq k} (\va B)_{\vi}\ \vx^{\vi}.
$$
Set $d = (\va B, q)$, $q' = d^{-1}q$ and $\vb' = d^{-1} (\va B)$. Then $(q' , \vb') =1$, so by \cite[Lemma 8, p. 54]{arhipov82}, we have
\begin{align*}
S(q, \va) & = d^2 \sum_{\vx \in [q']^2} e\biggbrac{\recip{q'}\sum_{0 < i_1 + i_2 \leq k} b'_{\vi}\ \vx^{\vi}}\\
		&  \ll_\eps d^2\ (q')^{2 - 1/k + \eps}
\end{align*}
It therefore remains to show that $d \ll_\Phi 1$.  Since $B$ has full row-rank, there exists a rational matrix $B'$ with 
\begin{equation}\label{B' inversion}
B B' = \Bigbrac{ I\ |\ 0},
\end{equation}
where $I$ is the identity matrix.  Clearly there exists a positive integer $m = O_\Phi(1)$ such that $mB'$ has only integer entries.  Hence we have 
 $$(ma_1, \dots, ma_N, 0, \dots, 0) = (\va B)(mB') \equiv 0 \bmod d.$$  So $d$ divides $m a_i $ for all $i$.  Since $d | q$ and $(q, \va) = 1$, we have $d |m$.  Thus $d \ll_\Phi 1$.\end{proof}

\begin{lemma}\label{integral bound}  For any $\eps > 0$ the auxiliary function $I(\vbeta, X)$ satisfies
\begin{equation}\label{product integral bound}
\begin{split}
I(\vbeta; X)  \ll_\eps X^2\Bigbrac{1 + X^{k_1}|\beta_1| + \dots + X^{k_N}|\beta_N|}^{-\recip{k}+\eps}
\end{split}
\end{equation}
\end{lemma}

\begin{proof}
Let $B$ be the matrix in the proof of Lemma \ref{complete sum bound}.  Changing variables in the integral $I(\vbeta; X)$ gives
\begin{equation}\label{I variable change}
I(\vbeta; X)  = X^2\int_0^1\int_0^1 e\biggbrac{\ \sum_{0 < i_1 + i_2 \leq k} X^{i_1 + i_2}(\vbeta B)_{\vi}\ \vgamma^{\vi}} \intd \vgamma.
\end{equation}
Let $\beta_i(X) = X^{k_i} \beta_i$, and $\alpha_{\vi} = (\vbeta(X) \cdot B)_{\vi} = X^{i_1 + i_2} (\vbeta B)_{\vi}$.  Then we can apply \cite[Lemma 2, p. 50]{arhipov82} to the double integral in \eqref{I variable change} to obtain
\begin{equation}
I(\vbeta; X) \ll_\eps X^2 \min\bigset{ 1 ,\  \abs{\valpha}_\infty^{-1/k + \eps}},
\end{equation}
where $\abs{\valpha}_\infty = \max_{\vi} |\alpha_{\vi}|$.  Using \eqref{B' inversion}, we have $\abs{\vbeta(X)}_\infty \ll_\Phi \abs{\valpha}_\infty$.  The result now follows.
\end{proof}

\begin{lemma}\label{generating asymptotic}  Let $q$ be a positive integer.  Then for any $\va \in \Z^n$ and $\valpha \in \T^n$
\begin{equation}\label{generating asymptotic formula}
f(\valpha) - V(\valpha; q , \va) \ll X\brac{q + \sum_{i=1}^n |q\alpha_i- a_i |X^{k_i }}.
\end{equation} 
\end{lemma}

\begin{proof}
Write $\valpha = \va/q + \vbeta$.  Sorting the sum $f(\valpha)$ into a sum over congruence classes modulo $q$, we have
 \begin{equation}\label{f_c identity}
f(\valpha) = \sum_{\vr \in [q]^2}e\brac{\va \cdot \vF(\vr)/q} \sum_{0\leq y_1 \leq \frac{X-r_1}{q}}\ \sum_{0\leq y_2 \leq \frac{X-r_2}{q}} e\brac{\vbeta \cdot \vF(q\vy+\vr)}.
\end{equation}
By the mean value inequality we have
$$
e\brac{\vbeta \cdot \vF(q\vy + \vr)} - q^{-2}\int_{qy_1}^{q(y_1+1)}\int_{qy_2}^{q(y_2+1)}e\brac{ \vbeta \cdot \vF( \vgamma)}\intd\vgamma \ll  \sum_{i=1}^N |q\beta_i|X^{k_i -1}.
$$
Summing over $\vr$ and $\vy$ shows that $f(\valpha)$ equals
\begin{equation}\label{halfway formula}
 q^{-2}\sum_{\vr \in [q]^2}e\brac{\va \cdot \vF(\vr)/q} \int_{0}^{X(r_1)}\int_{0}^{X(r_2)} e\brac{\vbeta \cdot \vF(\vgamma)}\intd\vgamma + \Oh{ X\sum_{i=1}^N |q\beta_i|X^{k_i }},
\end{equation}
where $X(r) = q\brac{\floor{\frac{X- r}{q}} +1}$.  Using the fact that $|X - X(r)| \leq q$ for all $r\in [q]$, we see that \eqref{halfway formula} is equal to
$$
 q^{-2}\sum_{\vr \in [q]^2}e\brac{\va \cdot \vF(\vr)/q} \int_{0}^{X}\int_{0}^{X} e\brac{\vbeta \cdot \vF(\vgamma)}\intd\vgamma + \Oh{ Xq + X\sum_{i=1}^N |q\beta_i|X^{k_i }},
$$
as required.
\end{proof}

With these bounds in hand, we are able to prove the theorem advertised at the start of this section.

\begin{proof}[Proof of the Theorem \ref{minor arc estimate theorem}]
Let $\tau = 2s\sigma + \Delta_s + \eps_1$, with $\eps_1$ sufficiently small (to be determined later).  The result is vacuous if $\sigma \leq 0$, so we may assume that $\Delta_s \leq 1/3$.  It then follows that $\tau < 1$.  By Dirichlet's principle, for each $i$ with $k_i \geq 2$ we can find co-prime integers $b_i, q_i$ with $1 \leq q_i \leq X^{k_i - \tau }$ and 
\begin{equation}\label{dirichlet approx}
|\alpha_i - b_i/q_i| \leq q_i^{-1} X^{\tau - k_i }.
\end{equation}
Let $C_1$ be the absolute constant in Lemma \ref{approximation to spacing lemma}.  Using \eqref{dirichlet approx}, notice that if $y, z \in [X]$ satisfy \eqref{single coefficient spacing bound}, then we have
\begin{equation}\label{hungry label}
\norm{L(y-z)b_i/q_i} \leq C_1X^{1-k_i} + LX^{1+\tau - k_i } q_i^{-1}.
\end{equation}
For each choice $y\in [X]$, the number of residue classes modulo $q_i$ containing some $z \in [X]$ satisfying \eqref{hungry label} is at most  $C_1X^{1-k_i} q_i + LX^{1+ \tau - k_i } +1$.  Let $D$ denote the maximum, over all $y \in [X]$, for the number of choices for $z \in [X]$ satisfying \eqref{single coefficient spacing bound}.  Then we have
\begin{align*}
D  & \leq (C_1X^{1-k_i} q_i + LX^{1+ \tau  - k_i} +1)(q_i^{-1} LX + 1)\\
		& \ll q_iX^{1- k_i} + Xq_i^{-1} + 1.
\end{align*}
Using Lemma \ref{approximation to spacing lemma}, we see that 
$$
|f(\valpha)| \ll X^2 \log(2X)^2 \brac{ X^{\Delta}(q_i X^{-k_i} + q_i^{-1} + X^{-1})}^{1/(2s)}.
$$
By the lower bound \eqref{f lower bound}, we have
\begin{align*}
X^{-2s\sigma - \Delta_s} & \ll (q_i X^{-k_i} + q_i^{-1} + X^{-1})\log(2X)^{4s}\\
		& \ll (X^{- \tau} + q_i^{-1})\log(2X)^{4s}.
\end{align*}
Since $ \tau> 2s\sigma + \Delta_s$, we must have
\begin{equation}
q_i \ll X^{2s \sigma + \Delta_s} \log(2X)^{4s}.
\end{equation}
Since $1 > 2s\sigma + \Delta_s$, this implies that the right-hand side of \eqref{hungry label} is strictly less than $q_i^{-1}$ (provided  $X \gg_\Phi 1$).  It follows that \eqref{hungry label} implies  $q_i | L(y-z)b_i$, which in turn implies $q_i | L(y-z)$, since $(q_i, b_i) =1$.  Hence it follows from the assumption \eqref{single coefficient spacing bound}  that $q_i | L(y-z)$ for all $i \in I_m$.  Let $Q_m$ denote the lowest common multiple of the set $\set{q_i : i \in I_m}$.  Then the number $D$ satisfies $D \ll XQ_m^{-1} + 1$.  Utilising Lemma \ref{approximation to spacing lemma} again, we obtain the bound
\begin{equation}\label{Q_m bound}
Q_m \ll X^{2s \sigma + \Delta_s} \log(2X)^{4s}.
\end{equation}
Let $Q = [Q_1, Q_2]$, so that $Q \ll X^{4s \sigma + 2\Delta_s} \log(2X)^{8s}$ by \eqref{Q_m bound}.  Since $\sigma$ is strictly less than $\tfrac{1- 3\Delta_s}{6s +3}$, we can (on taking $\epsilon_1$ sufficiently small) find a real $\mu$ satisfying
\begin{equation}\label{mu size}
\sigma < \mu < \trecip{2}\brac{1 - (4s + 1)\sigma - 2 \Delta - \tau }.
\end{equation}
As the space of linear binary homogeneous polynomials has dimension 2, there are at most two indices $i$ with $k_i =1$.  We can therefore use Dirichlet's principle to find a positive integer $1 \leq t \leq X^{2\mu}$, along with $a_i \in \Z$ $(k_i = 1)$ which are together co-prime to $t$ and such that 
\begin{equation}
|t(Q\alpha_i) - a_i| \leq X^{- \mu} \quad (k_i =1).
\end{equation}
Set $q = tQ$.  For $i$ with $k_i \geq 2$, let us define $a_i = (q/q_i)b_i$.  Then the $N$-tuple $\va = (a_1, \dots, a_N)$ satisfies $(q, \va) =1$ and 
\begin{equation}
|q \alpha_i - a_i| \ll \begin{cases} X^{- \mu} & (k_i =1),\\ X^{2\mu + 4s\sigma + 2 \Delta_s + \tau  - k_i}\log(2X)^{4s} & (k_i \geq 2). \end{cases}
\end{equation}
It thus follows from Lemma \ref{generating asymptotic} and \eqref{mu size} that
\begin{align*}
|f(\valpha) - V(\valpha; q, \va)| & \ll X^{2 - \mu} + X^{1 + 2\mu + 4s\sigma + 2\Delta_s + \tau}\log(2X)^{4s}\\
						& = o\Bigbrac{X^{2 - \sigma}}.
\end{align*}
Hence by the lower bound \eqref{f lower bound}, we have $|V(\valpha; q, \va)| \gg X^{2 - \sigma}$.  Combining this, together with Lemma \ref{complete sum bound} and Lemma \ref{integral bound} , we see that for any $\eps > 0$ we have
$$
q + \sum_{i=1}^N |q \alpha_i - a_i|X^{k_i} \ll X^{k\sigma + \frac{\eps}{2}}.
$$
Taking $X$ sufficiently large (in terms of $s$, $\eps$ and $\Phi$), we obtain the theorem.
\end{proof}

\section{The Asymptotic Formula}

In order to prove our density result, Theorem \ref{intro result 1}, we need to estimate the number of solutions to \eqref{non-singular solution equation 2} when the variables $\vx_j$ are restricted to the box $[X]^2$.

 \begin{definition}  Given a finite set $A \subset \Z^2$, write $R_{\vc, \Phi}(A)$ for the number of tuples $(\vx_1, \dots, \vx_s)$ in the set $A^s$ satisfying
 \begin{equation}
c_1 \Phi^{u,v}(\vx_1) + \dots + c_s \Phi^{u,v}(\vx_s) = 0 \quad (0 \leq u+v < k).
\end{equation}
When $A = [X]^2$, we simply write $R_{\vc, \Phi}(X)$.
\end{definition} 

The Hardy--Littlewood method gives an asymptotic for $R_{\vc, \Phi}(X)$, an asymptotic whose main term is a product of local densities, which we now define. 
\begin{definition}  Let $\Phi$ denote a binary form of degree $k$, differential dimension $N$ and differential degree $K$.  Let $\set{F_1, \dots, F_N}$ denote a maximal linearly independent subset of $\set{\Phi^{u,v} : 0 \leq u+v < k}$.  When $T > 0$, define $\lambda_T(y) = T\max\set{0, 1 - T|y|}$ and 
$$
\mu_{T}= \mu_T(\vc) = \int_{[0,1]^{2s}} \lambda_T\Bigbrac{\sum_{j=1}^s c_j F_1(\vgamma_j)} \dotsm \lambda_T\Bigbrac{\sum_{j=1}^s c_j F_N(\vgamma_j)}\intd \vgamma
$$
The limit $\sigma_\infty = \sigma_\infty(\vc) = \lim_{T \to \infty} \mu_{T}$, when it exists, is called the \emph{real density}.  Given a natural number $q$, we write
$$
M(q) = M(q; \vc) = \hash\Bigset{ \vx \in (\Z/q\Z)^{2s} : \sum_{j=1}^s c_j \vF(\vx_j) \equiv 0 \pmod q}.
$$ For each prime $p$, the limit
\begin{equation}\label{T(p) relation}
\sigma_p(\vc) = \lim_{H \to \infty} p^{- H(2s-N)} M(p^H),
\end{equation}
provided it exists, is called the \emph{$p$-adic density}.
\end{definition}

The purpose of this section is to prove the following asymptotic formula. 

\begin{theorem}\label{asymptotic formula}
Let $\Phi$ be a non-degenerate binary form of degree $k$, differential dimension $N$ and differential degree $K$. Suppose that
\begin{equation}\label{number of variables}
s \geq kN(\log K + \log \log K + 26).
\end{equation}
Then there exists $\delta > 0$ such that for any choice of non-zero integers $c_1, \dots, c_s$ we have
\begin{equation}\label{non legit asymp}
R_{\vc, \Phi}(X) = \sigma_\infty\Bigbrac{\prod_p \sigma_p} X^{2s - K} + O(X^{2s - K - \delta})
\end{equation}
Suppose in addition that $\vc$ is a non-singular choice of coefficients for $\Phi$.  Then
\begin{equation}
 \sigma_\infty \prod_p \sigma_p > 0.
\end{equation}
\end{theorem}

\begin{remark}
The $O(1)$ constant in \eqref{number of variables} can certainly be lowered from 26 if one is willing to implement the results of \S \ref{Weyl-type section} more optimally.
\end{remark}

The proof of Theorem \ref{asymptotic formula} proceeds by the usual Hardy--Littlewood dissection into major and minor arcs. 

\begin{definition}
Given a tuple of integers $\va = (a_1, \dots, a_N)$ and $q \in \N$, define the \emph{major arc} centred at $\va/q$ to be the set
$$
\M(q, \va) = \Bigset{ \valpha \in \T^N : \norm{\alpha_i - a_i/q} \leq q^{-1}X^{\recip{4} - k_i}\quad (1 \leq i \leq N)}.
$$
Define the \emph{major arcs} $\M$ to be the union of the sets $\M(q,\va)$ with $1 \leq q \leq X^{1/4}$ and $\va \in [q]^n$ subject to $(q, \va) = 1$.  Define the \emph{minor arcs} to be the complement $\m = \T^N \setminus \M$.
\end{definition}

One can show that for $X \gg 1$ the major arcs are disjoint.  We can therefore define the function $V(\valpha)$ to equal $V(\valpha; q, \va)$ when $\valpha \in \M(q, \va) \subset \M$, and equal $0$ otherwise.

\begin{lemma}\label{V major arc asymptotic}  Whenever $s \geq k(N+1) +1$ there exists $\delta > 0$ such that 
\begin{equation}
\int_\M V(c_1\valpha) \dotsm V(c_s\valpha) \intd\valpha = \mathfrak{J} \mathfrak{S} X^{2s - K} +O( X^{2s - K - \delta}),
\end{equation}
where
\begin{equation}
\mathfrak{J} = \mathfrak{J}(\vc)= \int_{\R^N} \int_{[0,1]^{2s}} e\Bigbrac{ \vbeta \cdot \sum_{j=1}^s c_j\vF(\vgamma_j)} \intd\vgamma \intd\vbeta
\end{equation}
and
\begin{equation}
\mathfrak{S} =\mathfrak{S}(\vc) = \sum_{q=1}^\infty q^{-2s}\twosum{\va \in [q]^N}{(q, \va) =1}  S(q,c_1\va) \dotsm S(q, c_s\va).
\end{equation}
\end{lemma}

\begin{proof}  
Define $A(q)$ to be the sum
\begin{equation}\label{A(q) defn}
A(q) = \twosum{\va \in [q]^N}{(q, \va) =1 } q^{-2s} S(q, c_1\va)\dotsm S(q, c_s\va),
\end{equation}
and let $\mathcal{I}(\vbeta; X)$ denote the product
$$
\mathcal{I}(\vbeta; X) = I(c_1\vbeta ; X) \dotsm I(c_s\vbeta; X).
$$
Then by disjointness of the major arcs, and a change of variables $\beta_i = (\alpha_i - a_i/q) X^{k_i}$, we have
\begin{equation}\label{original variable change}
\int_\M V(c_1\valpha) \dotsm V(c_s\valpha) \intd\valpha  = X^{-K}\sum_{1 \leq q \leq X^{1/4}} A(q) \int_{\mathcal{B}_q} \mathcal{I}( \beta_1X^{-k_1}, \dots, \beta_NX^{-k_N}; X) \intd\vbeta,
\end{equation}
where $\mathcal{B}_q = \prod_{1 \leq i \leq N} [-q^{-1}X^{1/4}, q^{-1}X^{1/4}]$.  Let $\delta_1 = \trecip{5}(\tfrac{s}{kN} - 1) > 0$.  By Lemma \ref{integral bound} and the AM--GM inequality we have
\begin{equation}\label{product integral bound}
\mathcal{I}( \beta_1X^{-k_1}, \dots, \beta_NX^{-k_N}; X) \ll X^{2s}\prod_{i=1}^N \bigbrac{1 + |\beta_i|}^{-\frac{s}{kN} + \delta_1}.
\end{equation}
It follows that 
\begin{align*}
\int_{\R^N \setminus \mathcal{B}_q} \mathcal{I}( \beta_1X^{-k_1},& \dots, \beta_NX^{-k_N}; X) \intd\vbeta \\ & \ll X^{2s} \int_{X^{1/4}}^\infty x^{-(1+4\delta_1)}\intd x \brac{\int_\R \recip{(1 + |x|)^{1 + 4\delta_1}}\intd x}^{s-1}\\
			& \ll_{s, \delta_1} X^{2s - \delta_1}.
\end{align*}
Combining this with another change of variables, we have  
\begin{equation}\label{singular integral error}
 \int_{\mathcal{B}_q} \mathcal{I}( \beta_1X^{-k_1}, \dots, \beta_NX^{-k_N}; X) \intd\vbeta = \mathfrak{J} X^{2s } + O(X^{2s - \delta_1}).
\end{equation}
Set $\delta_2 = \recip{5}(\frac{s}{k} - N - 1) > 0$.  By Lemma \ref{complete sum bound}
$$
A(q) \ll q^{N - \frac{s}{k} + \delta_2} =  q^{-1 - 4\delta_2}.
$$
Thus
\begin{align*}
\sum_{q > X^{1/4}} |A(q)|  \ll \sum_{q > X^{1/4}} q^{-1-4\delta_2}  \ll_{\delta_2} X^{- \delta_2}.
\end{align*}
It follows that 
\begin{equation}\label{singular series error}
\sum_{1 \leq q \leq X^{1/4}} A(q) = \mathfrak{S} + O(X^{-\delta_2}).
\end{equation}
Combining \eqref{original variable change}, \eqref{singular integral error} and \eqref{singular series error}, we obtain the result.
\end{proof}

\begin{lemma}\label{minor arc integral lemma}
There exist positive integers $r$ and $t$ such that for any $s \geq r + 2Mt$ there exists $\tau> 0$ such that
\begin{equation}\label{minor arc estimate equation}
\int_\m |f(\valpha)|^{s} \intd\valpha \ll X^{2s - K - \tau}.
\end{equation}
Moreover, one can ensure that 
$$
r + 2tM \leq kN(\log K + \log\log K + 26).
$$
\end{lemma}

\begin{proof}

Let us first find a large value for the expression $\tfrac{1-3\Delta_s}{6s + 3}$ occurring in Theorem \ref{minor arc estimate theorem}.  Setting $s_0 = M\ceil{k \log(21K)}$, by Theorem \ref{intro VMVT} we have
$$
\Delta_{s_0} < Ke^{- \ceil{k \log (21 K)}/k} \leq \trecip{21}.
$$ 
Therefore 
$$
\tfrac{1-3\Delta_{s_0}}{6s_0 + 3} > \trecip{7 s_0 + (7/2)} : = \sigma, \ \text{say}.
$$
Since $\Phi$ is non-degenerate of degree at least two, Lemma \ref{no one dim space} guarantees that $\Phi$ has two linearly independent derivatives of degree one.  This implies  that $N \geq 3$ and $K \geq 4$.  Hence
\begin{equation}\label{k sigma bound}
k \sigma = \tfrac{k}{7s_0 + (7/2)} \leq \trecip{M \log(21 K)} \leq \trecip{2.\log(21.4)} < \trecip{8}.
\end{equation}
Let $\valpha \in \m$ and suppose that 
\begin{equation}
|f(c_j\valpha) | \geq X^{2 - \sigma}.
\end{equation}
Provided $X$ is sufficiently large, it follows from Theorem \ref{minor arc estimate theorem} and \eqref{k sigma bound} that there exists $q \in \N$ and integers $a_1, \dots, a_N$, with $|q (c_j\alpha_i )- a_i| \leq X^{1/8}$ and $q \leq X^{1/8}$.  For $X \gg_{\vc} 1$ sufficiently large, we have $|c_j|q \leq |c_j|X^{1/8} \leq X^{1/4}$, so $\valpha \in \M(c_jq, \vb) \subset \M$, a contradiction.  Hence we must in fact have
\begin{equation}\label{sigma saving}
|f(c_j\valpha) | \leq X^{2 - \sigma}.
\end{equation}

Set
\begin{equation}\label{t and r def}
t = \ceil{k \log(K\log K)} \quad \text{and} \quad r = \ceil{\sigma^{-1}Ke^{-t/k}},
\end{equation}
and let $\Delta_{tM}$ be an admissible exponent for $(tM, \Phi)$.  It suffices to prove \eqref{minor arc estimate equation} for $s = r + 2t M$.  By \eqref{sigma saving}, H\"{o}lder's inequality and Theorem \ref{intro VMVT}, we have
\begin{align*}
\int_{\m} f(c_1\valpha) \dotsm f(c_s\valpha)\intd\valpha & \leq X^{2r - r\sigma } \oint |f(\valpha)|^{2tM} \intd\valpha\\
										& \ll X^{2s - K -(r\sigma - \Delta_{tM})}.
\end{align*}
Since $r \sigma > \Delta_{tM}$, we obtain \eqref{minor arc estimate equation}.  

It remains to show that $r + 2tM \leq k N(\log K + \log\log K + 21)$.  Using the fact that  $k \geq 2$, $N \geq 3$, $K \geq 4$ and $K \leq \min\set{ N^2, k^3}$, we have 
\begin{align*}
2tM & \leq t(N + 1) \\ & \leq k(N+1) \log(K\log K) + N +1\\
				& \leq kN(\log K + \log\log K +4)
\end{align*}
and
\begin{align*}
 r  & \leq Ke^{-t/k} \sigma^{-1} + 1 \\
 	& \leq \frac{7kM\log(21K) + 7 M + (7/2)}{\log K} +1\\
	& \leq 22kN.
\end{align*}\end{proof}

\begin{proof}[Proof of Theorem \ref{asymptotic formula}]
Let $r$ and $t$ be defined as in the proof of Lemma \ref{minor arc integral lemma}.  We deduce the theorem under the weaker assumption that $s \geq r + 2tM$.  From this assumption, it follows that $s = u + 2v M$, where $u \geq 1$ and  $v \geq kM(K\log(K\log K) + 6)$.  Therefore $\Delta_{vM} \leq e^{-6} < 3/4$.  Combining this with Theorem \ref{intro VMVT}, Lemma \ref{generating asymptotic} and H\"{o}lder's inequality, we see that there exists $j \in [s]$ such that
\begin{align*}
\int_{\M}  ( f(c_1\valpha) \dotsm f(c_s\valpha) &- V(c_1\valpha) \dotsm V(c_s\valpha)) \intd\valpha\\ & \ll X^{2u - \frac{3}{4}}\Bigbrac{ \oint |f(\valpha)|^{2vM}\intd\valpha + \int_{\M} |V(c_j\valpha)|^{2vM} \intd\valpha}\\
& \ll X^{2s - K + \Delta_{vM} - \frac{3}{4}} + X^{2s - K - \frac{3}{4}}\\
& \ll X^{2s - K - \tau_1}, \quad \text{say}.
\end{align*}
Using this, together with Lemma \ref{V major arc asymptotic} and Lemma \ref{minor arc integral lemma}, we see there exists $\tau_2 > 0$ such that
\begin{equation}
\oint f(c_1\valpha) \dotsm f(c_s\valpha) \intd\valpha = \mathfrak{S}\mathfrak{J} X^{2s - K} + O(X^{2s - K - \tau_2})
\end{equation}

It remains to show that $\mathfrak{J} = \sigma_{\infty}$, that $\mathfrak{S} = \prod_p \sigma_p$ and that these quantities are positive under the appropriate non-singularity hypotheses.  To prove $\mathfrak{J} = \sigma_\infty$ we use a method of Schmidt \cite{schmidt82}, as described by Parsell \cite{parsell02}.  For a positive real $T$, define
$$
K_T(\beta) = \brac{ \frac{\sin(\pi \beta T^{-1})}{\pi \beta T^{-1}}}^2, \qquad \mathcal{K}_T(\vbeta) = K_T(\beta_1) \dotsm K_T(\beta_n).
$$
It follows from Baker \cite[Lemma 14.1]{baker86} that
\begin{equation}\label{transform calculation}
\begin{split}
\hat{K}_T(y) & = \int_\R K_T(\beta) e( - \beta y ) d\beta\\
			& = T \max\set{0, 1- T|y|}.
\end{split}
\end{equation}
In particular, this Fourier transform is always non-negative.  Write
$$
I(\vbeta) = I(\vbeta; 1) = \int_{[0,1]^2} e\brac{\vbeta \cdot \vF(\gamma)}d\vgamma \qquad \text{and} \qquad \mathcal{I}(\vbeta) = I(c_1\vbeta)\dotsm I(c_s\vbeta).$$
By Fubini's theorem, we have that 
\begin{align*}
\mu_T & = \int_{[0,1]^{2s}} \hat{K}_T\Bigbrac{\sum_{i=1}^sc_iF_1(\vgamma_i)} \dotsm  \hat{K}_T\Bigbrac{\sum_{i=1}^sc_iF_N(\vgamma_i)} \intd\vgamma\\
		& =  \int_{\R^N} \mathcal{K}_T(\vbeta) \mathcal{I}(\vbeta) \intd\vbeta
\end{align*}
Lemma \ref{integral bound} and the AM--GM inequality ensure that, for any $\eps > 0$, we have the bound
$$
\mathcal{I}(\vbeta) \ll_{\eps} \prod_{1\leq i \leq N}(1+ |\beta_i|)^{-\frac{s}{kN} + \eps},
$$
and a simple estimate reveals that
\begin{equation}\label{K bound}
1 - \mathcal{K}_T(\vbeta) \ll \min\set{1, |\vbeta|^2 T^{-2}}.
\end{equation}
Therefore
\begin{align*}
\mathfrak{J} - \mu_T  & \ll \int_{\R^N}\min\set{1, |\vbeta|^2 T^{-2}} \prod_{1\leq i \leq N}(1+ |\beta_i|)^{-\frac{s}{kN} + \eps} \intd\vbeta\\
				& \ll \int_{|\vbeta| > T^{\recip{3N}} }\prod_{1\leq i \leq N}(1+ |\beta_i|)^{-\frac{s}{kN} + \eps} \intd\vbeta + 
				\int_{|\vbeta| \leq T^{\recip{3N}}}|\vbeta|^2 T^{-2}\intd\vbeta.
\end{align*}
Using $s > kN$, we see that
$$
\mathfrak{J} = \lim_{T \to \infty} \mu_T = \sigma_\infty.
$$

Next, let us suppose that there exists a non-singular real solution to \eqref{non-singular solution equation 2}.  Writing $\vP(\vx)$ for $\sum_{i=1}^s c_i \vF(x_{2i-1}, x_{2i})$, it follows that there is some $\vxi \in \R^{2s}$ for which $ \vP(\vxi) = 0$, along with $S \subset [2s]$ such that $|S| = N$ and
$$
\det\brac{ \frac{ \partial P_i}{\partial \xi_j}(\vxi)}_{\substack{1 \leq i \leq N\\ j \in S}} \neq 0.
$$
The translation-dilation invariance of \eqref{non-singular solution equation 2} ensures that we may assume that $ \vxi \in (0, 1)^{2s}$.    Let $[2s]\setminus S = \set{l(N+1), \dots, l(2s)}$.  Define the function $\rho : \R^{2s} \to \R^{2s}$ by 
$$
\rho_i(\vgamma) = \begin{cases} P_{i}(\vgamma) & \text{if } 1 \leq i \leq N,\\
						\gamma_{l(i)} & \text{if } N < i \leq 2s. \end{cases}
$$
Let $\veta = \rho(\vxi)$, so that $\eta_i = 0$ for $1 \leq i \leq N$.  Notice that 
$$
|\det \rho'(\vgamma)| = \abs{\det\brac{ \frac{ \partial P_i}{\partial \xi_j}(\vgamma)}_{1 \leq i \leq N,\, j \in S} }.
$$
By the Inverse Function Theorem, there exists an open set $U\subset [0,1]^{2s}$ which contains $\vxi$ and an open set $V$ containing $\veta$ such that $\rho$ is a homeomorphism from $U$ to $V$.  
Define the constant $C_1 = C_1(\Phi, s)$ by
$$
C_1 = \max_{\vgamma \in [0,1]^{2s}} \abs{\det\brac{ \frac{ \partial P_i}{\partial \xi_j}(\vgamma)}_{1 \leq i \leq N,\, j \in S} }.
$$
Using positivity of the Fourier transform $\hat{K}_T$ and the fact that $U \subset [0,1]^{2s}$, we have 
$$
\int_{[0,1]^{2s}} \hat{K}_T(P_1(\vgamma)) \dotsm  \hat{K}_T(P_N(\vgamma)) \intd\vgamma \geq \int_{U} \hat{K}_T(P_1(\vgamma)) \dotsm  \hat{K}_T(P_N(\vgamma)) \intd\vgamma.
$$  
This latter integral is in turn bounded below by 
$$
 \recip{C_1} \int_{U} \hat{K}_T(\rho_{1}(\vgamma)) \dotsm  \hat{K}_T(\rho_{N}(\vgamma)) |\det \rho'(\vgamma)| \intd\vgamma.
$$
By a change of variables this equals 
$$
\recip{C_1} \int_{V} \hat{K}_T(\zeta_{1}) \dotsm  \hat{K}_T(\zeta_{N}) \intd\vzeta.
$$

Since $V$ is defined independently of $T$, there exists $\eps = \eps(\Phi, s) > 0$ such that if $|\zeta_i - \eta_i| \leq \eps$ $(1 \leq i \leq 2s)$, then $\vzeta \in V$.  Let $W_T$ denote the set of $\vzeta$ for which $|\zeta_{i} | \leq (2T)^{-1}$ $(1\leq i \leq N)$ and $|\zeta_i - \eta_i| \leq \eps$ $(i > N)$.  Then for $T \geq (2\eps)^{-1}$, the set $W_T$ is contained in $V$.  Moreover, for $\vzeta \in W_T$ and $1\leq i \leq N$ we have $\hat{K}_T(\zeta_{i}) \geq T/2$.  Therefore
\begin{align*}
 \int_{V} \hat{K}_T(\zeta_{1}) \dotsm  \hat{K}_T(\zeta_{N}) \intd\vzeta &\geq \int_{W_T} \hat{K}_T(\zeta_{1}) \dotsm  \hat{K}_T(\zeta_{N}) \intd\vzeta\\
 & \geq \meas (W_T) \frac{T^N}{2^N}\\
 & \geq T^{-N} (2\eps)^{2s-N}\frac{T^N}{2^N}\\
 & \gg_{\Phi, s} 1.
 \end{align*}
 Hence $\mu_T \gg_{\Phi, s} 1$ for all sufficiently large $T$.

Let us now turn to the singular series $\mathfrak{S}$.  Recalling \eqref{A(q) defn}, for each prime $p$ define 
$$
T(p) = \sum_{h=0}^\infty A(p^h),
$$
By Lemma \ref{complete sum bound} this series is absolutely convergent for $s > k(N+1)$.  

Let $q$ and $r$ be coprime positive integers.  By Euclid's algorithm, any pair $\vx$ of residues modulo $qr$ can be represented uniquely in the form $r\vy + q \vz$ with $\vy \in [q]^2$ and $\vz \in [r]^2$.  It follows that for $\va, \vb \in \Z^N$ we have $S(qr, r\va + q \vb) = S(q, \va)S(r, \vb)$.  Again, by Euclid's algorithm, each $N$-tuple $\va'$ of residues modulo $qr$ with $(\va', qr) = 1$ can be represented uniquely in the form $r\va + q \vb$ with $\va \in [q]^N,\ (\va, q) =1$ and $\vb \in [r]^N,\ (\vb, r) = 1$.  A similar argument therefore gives $A(qr) = A(q)A(r)$.

Let $p_1, \dots, p_m$ denote the primes bounded above by $X$.  Using multiplicativity of $A(q)$, together with Lemma \ref{complete sum bound} and the fact that $\frac{s}{k} -N = 1 +2\eps$ for some $\eps >0$, we have
\begin{align*}
\prod_{p \leq X} T(p) - \mathfrak{S} & = \sum_{h_1=0}^\infty \dots \sum_{h_m = 0}^\infty A(p_1^{h_1}\dotsm p_m^{h_m}) - \sum_{q =1}^\infty A(q)\\
& \ll \sum_{q > X}\abs{A(q)}\\
& \ll \sum_{q> X} q^{N + \eps - s/k} \to 0 \qquad \text{as $X\to \infty$}.
\end{align*}

By orthogonality
\begin{align*}
M(p^H) & = \sum_{\vx \in [p^H]^{2s}} p^{-NH} \sum_{\va \in [p^H]^N} e\brac{ \va \cdot \vP(\vx)/p^H}\\
		& = p^{-HN} \sum_{\va \in [p^H]^N} \prod_{j=1}^s S(p^H, c_j \va).
\end{align*}
Partitioning the sum over $\va$ according to the value of $(p^{H}, \va)$, we see that $M(p^H) $ is equal to
\begin{align*}
 p^{H(2s - N)}\sum_{h = 0}^H A(p^h).
\end{align*}
It follows that $\mathfrak{S} = \prod_p \sigma_p$.

To show positivity of $\mathfrak{S}$, we begin with the following result from elementary linear algebra.
\begin{lemma}\label{ph lemma}
Let $h, H$ be non-negative integers with $H \geq h+1$.  Suppose that $A$ is an $n \times n$ integer matrix with $p^h || \det A$.  Then the image $\set{ A \cdot \vx: \vx \in( \Z/ p^H \Z)^n}$ contains the subgroup $\set{p^h \vy : \vy \in( \Z/ p^H \Z)^n}.$
\end{lemma}

For a proof of this lemma, let $A_{ij}$ denote the $ij$-minor of $A$, obtained from $A$ by deleting the $i$th row and $j$th column.  We define the adjunct matrix of $A$ by
$$
\adj(A) = \brac{ (-1)^{i+j} A_{ji}}_{1\leq i, j\leq n}.
$$
Then we have the identity
\begin{equation}\label{adjunct identity}
A\cdot \adj(A)= \det(A) I_n.
\end{equation}
Since $p^{h} || \det(A)$, we have $\det(A) = u p^h$ where $u$ is a unit in $\Z/ p^{H} \Z$.  Let $\vy \in (\Z/ p^{H} \Z)^n$.  Then
\begin{align*}
p^h \vy & = \det(A) (u^{-1} \vy)\\
		& = A\cdot ( u^{-1} \adj(A)\cdot\vy),
\end{align*} as required.  This completes the proof of Lemma \ref{ph lemma}.

Given a subset $S \subset [2s]$, define the Jacobian matrix 
$$
J_{\vP}(\vx; S) = \brac{ \frac{\partial P_i}{\partial x_j} (\vx)}_{1 \leq i \leq N,\, j \in S}.
$$
When $|S| = N$ we define $\Delta_{\vP}(\vx; S)$ to be the determinant of $J_{\vP}(\vx; S)$.  Given a positive integer $h$, let $\mathcal{B}_{h}(p^H)$ denote the set of $\vx \in (\Z/ p^H \Z)^{2s}$ with $\vP(\vx) \equiv 0 \bmod p^H$ and for which there exists $S \subset [2s]$ with $|S| = N$ and $p^h ||  \Delta_{\vP}(\vx; S)$.  The following claim is a version of Hensel's lemma.
\begin{claim}
For $H \geq 2h+1$ we have the bound
\begin{equation}\label{nonsing gamma count}
\abs{\mathcal{B}_{h}(p^{H+1})} \geq p^{(2s-N)}\abs{\mathcal{B}_{h}(p^{H})}.
\end{equation}
\end{claim}

Fix $\vx \in (\Z/ p^H \Z)^{2s}$ with $\vP(\vx) \equiv 0 \bmod p^H$ and $S \subset [2s]$ with $|S| = N$ and $p^h ||  \Delta_{\vP}(\vx; S)$.  For each $j\notin S$ choose $y_j \in [p]$ and define
$$
z_j = \begin{cases} x_j & (j \in S),\\
				x_j + p^H y_j & (j \notin S).\end{cases}
$$  Let $\vw \in \Z^{2s}$ be subject to the condition that $w_j = 0$ if $j \notin S$.
By the binomial theorem and the fact that $2(H - h) \geq H+1$, we have
\begin{equation}\label{congruence we want} 
\vP(\vz + p^{H-h}\vw)  \equiv \vP(\vz) + p^{H-h} J_{\vP}(\vz; S) \cdot (w_j)_{j \in S}\pmod{ p^{H+1}}.
\end{equation}
Since $\vP(\vz) \equiv \vP(\vx) \equiv 0 \mod p^{H}$, we see that $p^h$ divides every entry in the $N$-tuple of integers $\vP(\vz)/p^{H-h}$.  Hence
$$
-\vP(\vz)/p^{H-h} \in \set{ p^h \vy : \vy \in ( \Z/ p^{h+1} \Z)^N}.
$$
Notice that $\Delta_{\vP}(\vz; S) \equiv \Delta_{\vP}(\vx; S) \mod p^{h+1}$, and so $p^h || \Delta_{\vP}(\vz; S)$.  Therefore, by Lemma \ref{ph lemma}, for each $j \in S$ we can find $w_j \in \Z/ p^{h+1} \Z$ so that 
$$
J_{\vP} (\vz; S)\cdot (w_j)_{j \in S} \equiv -\vP(\vz)/p^{H-h} \pmod{p^{h+1}}.
$$
Moreover, since $H- h \geq h+1$, we have $\Delta_{\vP}(\vz + p^{H-h}\vw; S) \equiv \Delta_{\vP}(\vx; S) \mod p^{h+1}$.  Hence
$$
\vz + p^{H-h}\vw \in \mathcal{B}_{h}(p^{H+1}).
$$
Since $w_j = 0$ if $j \notin S$, we see that for each choice of $\vz$, the sum $\vz + p^{H-h}\vw$ gives a unique element of $
\mathcal{B}_{h}(p^{H+1})$.  As there are $p^{2s-N}$ choices for $\vz$ for each choice of $\vx \in \mathcal{B}_{h}(p^{H})$, the claim follows.

Suppose there exists $\vx \in \Q_p^{2s}$ such that $\vP(\vx) = 0$ and $S \subset [2s]$ such that the Jacobian matrix $J_\vP(\vx; S)$ is non-singular over $\Q_p$.  By homogeneity of the $P_i$ we may assume all the entries of $\vx$ are $p$-adic integers.  Hence there exists a non-negative integer $h$ such that $|\Delta_{\vP}(\vx; S)|_p = p^{-h}$.  Take any $\vy \in \Z^t$ such that $\vy \equiv \vx \mod p^{2h+1}$.  Then $p^h || \Delta_{\vP}(\vy; S)$ and $\vP(\vy) \equiv 0 \mod p^{2h+1}$, so $\vy \in \mathcal{B}_{h}(p^{2h+1})$.  In particular, $|\mathcal{B}_{h}(p^{2h+1})|\geq 1$.  Iterating the bound \eqref{nonsing gamma count} obtained in the previous lemma, we have established that there exists a non-negative integer $h = h(\Phi, p)$ such that for all $H \geq 2h+1$ we have the lower bound
\begin{equation}\label{MF lower bound}
|\mathcal{B}_{h}(p^H)| \geq p^{(2s-N)(H- 2h -1)}.
\end{equation}
Clearly $M(p^H) \geq \mathcal{B}_h(p^H)$, so inputting this into the relation \eqref{T(p) relation}, we obtain
\begin{align*}
T(p) & = \lim_{H \to \infty} p^{-H(2s-N)}M(p^H)\\
	& \geq p^{-(2s-N)(2h+1)}\\
	& \gg_{s, \Phi, p} 1.
\end{align*}

The absolute convergence of the product $\mathfrak{S} = \prod_p T(p)$, with all $T(p)$ positive, implies that
$$
\lim_{X\to \infty} \prod_{p > X} T(p) = 1.
$$
In particular, there exists $X_0$ such that for all $p > X_0$ we have
$$
\prod_{p > X_0} T(p) > 1/2.
$$
Since $T(p) > 0$ for all $p \leq X_0$, we also have $\prod_{p\leq X_0} T(p) > 0$.  Therefore
$$
\mathfrak{S} = \prod_{p\leq X_0} T(p) \prod_{p > X_0} T(p) > 0.
$$
\end{proof}

\section{Density bounds for solution-free sets}

This section is dedicated to the proof of our main theorem.

\begin{theorem}\label{full density theorem}
Let $\Phi \in \Z[x,y]$ be a binary form of degree $k\geq 2$, differential dimension $N$ and differential degree $K$, and let $\vc \in \Z^s$ be a non-singular choice of coefficients for $\Phi$ with $c_1 + \dots + c_s = 0$ .  Suppose that $s \geq kN(\log K + \log\log K + 27)$.  Then any set $A \subset [X]^2$ containing only diagonal solutions $(\vx_1, \dots, \vx_s) \in A^s$ to the system of equations
\begin{equation}\label{multi equation}
 c_1 \Phi^{u,v}(\vx_1) + \dots + c_s \Phi^{u,v}(\vx_s) = 0 \qquad (u+v \geq 0),
\end{equation}
satisfies the bound
\begin{equation}\label{density bound 2}
 |A| \ll X^2\brac{\log\log X}^{-1/(s-1)}.
\end{equation}
Here the implicit constant depends only on $\vc$ and $\Phi$.
\end{theorem}

\begin{remark} We will prove Theorem \ref{full density theorem} under the assumption that $\Phi$ is non-degenerate.  The degenerate case follows from the same argument, but the superior bounds available in the standard Vinogradov mean value theorem ensure that, in this case, the lower bound on the number of variables required can be decreased.
\end{remark}

In order to prove Theorems \ref{full density theorem} it is useful to work with translates of sets of the form $[X]^2$.  We define a \emph{half-open square} to be a subset of $\R^2$ of the form
$$
Q = \vx + (0, X]^2,
$$
and call $X$ the \emph{side-length} of $Q$. Let us write $[Q]$ to denote the set of integer points in $Q$, namely $[Q] = Q \cap \Z^2$.

We reduce the proof of Theorem \ref{full density theorem} to the following density increment result.

\begin{lemma}\label{increment lemma}
Given the assumptions in Theorem \ref{full density theorem}, there exist absolute constants $\tau = \tau(k)$, $C = C(\vc, \Phi)$ and $c = c(\vc, \Phi) > 0$ such that for any $\delta > 0$ and any real 
$X \geq \exp(C/\delta)$, if $Q \subset \R^2$ is a half-open square with side-length $X$ and $A \subset [Q]$ satisfies $|A| = \delta |[Q]|$ and
\begin{equation}\label{A conditions}
R_{\vc, \Phi}(A)\leq c\delta^sX^{2s - K},
\end{equation}
then there exists a half-open square $Q_1$ with side-length at least $2^{-k}X^{\tau}$, along with $q \in \N$ and $\vr \in \Z^2$, such that 
\begin{equation}\label{square increment}
|A\cap (\vr + q \cdot [Q_1])| \geq (\delta + c\delta^s) |[Q_1]|.
\end{equation}
\end{lemma}

\begin{proof}[Proof that Lemma \ref{increment lemma} implies Theorem \ref{full density theorem}]

Let us suppose that $A \subset [X]^2$ contains only diagonal solutions to \eqref{multi equation} and let $\tau$, $C$ and $c$ be as in Lemma \ref{increment lemma}.  We aim to construct a sequence of quadruples $(Q_i, A_i, X_i, \delta_i)$ satisfying all of the following conditions. 
\begin{enumerate}[(i)]
\item $Q_i$ is a half-open square of side-length $X_i$.
\item $A_i \subset [Q_i]$ with $A_i = \delta_i |[Q_i]|$.
\item $A_i$ contains only diagonal solutions to \eqref{multi equation}.
\item $X_{i+1} \geq 2^{-k}X_i^{\tau}$.
\item $\delta_{i+1} \geq \delta_i + c\delta_i^s$.
\end{enumerate}
Taking $Q_0 = Q$, $A_0 = A$, $X_0 = X$ and $\delta_0 = |A_0|/|[Q_0]|$, we have our initial quadruple.  Let us suppose we have constructed $(Q_j, A_j, X_j, \delta_j)$ for all $1 \leq j \leq i$.  In order to apply Lemma \ref{increment lemma}, we must estimate $R_{\vc, \Phi}(A_i)$.  First notice that $A_i^s$ contains exactly $|A_i|$ solutions to \eqref{multi equation} with $\vx_1 = \dots = \vx_s$.  Any other solution counted by $R_{\vc, \Phi}(A_i)$ must have all $\vx_j$ contained on some affine line $L$, where $ |L \cap A_i| \geq 2$.  Any 2-set $\set{\vx, \vy} \subset A_i$ is contained in exactly one affine line $L$.  Letting $\mathcal{L}$ denote the set of affine lines which intersect $A_i$ in at least two places, we therefore have
\begin{align*}
R_{\vc, \Phi}(A_i) & \leq |A_i| + \sum_{L \in \mathcal{L}} |L \cap [Q_i]|^s\\
			& \leq |A_i| + \binom{|A_i|}{2}\max_{L \in \mathcal{L}} |L \cap [Q_i]|^s\\
			& \ll X_i^4\max_{L \in \mathcal{L}} |L \cap [Q_i]|^s.
\end{align*}

The set of integer points in $L \cap Q_i$ projects injectively onto either the $x$ or $y$ axis, with image equal to a set of integer points contained in a subinterval of length $X_i$.  Hence $|L \cap Q_i| \leq X_i +1 $.  Thus
$$
R_{\vc, \Phi}(A_i) \ll  X_i^{s + 4}.
$$ 
Our assumption on the size of $s$ certainly ensures that $2s -K > s+4$, hence taking $C$ sufficiently large in the assumption 
\begin{equation}\label{second X_i lower bound}
X_i \geq \exp(C/\delta_i),
\end{equation} 
certainly implies that
$$
R_{\vc, \Phi}(A_i) \leq 2^{s+3} X_i^{s+4} \leq c \delta_i^{s} X_i^{2s - K}.
$$
Assuming \eqref{second X_i lower bound}, we can therefore employ Lemma \ref{increment lemma} to obtain a half-open square $Q_{i+1}$ of side-length $X_{i+1} \geq 2^{-k} X_i^{\tau}$, together with $q$ and $\vr$ such that
$$
|A_i\cap(\vr + q \cdot [Q_{i+1}])| \geq (\delta_i + c\delta_i^{s})|[Q_{i+1}]|.
$$
Let us set $A_{i+1} = \set{ \vx \in [Q_{i+1}] : \vr + q \vx \in A_i}$ and $\delta_{i+1} = |A_{i+1}|/ |[Q_{i+1}]|$.  The fact that $c_1 + \dots + c_s = 0$ means the system \eqref{multi equation} is translation-dilation invariant.  Using this, it follows that if $(\vx_1, \dots, \vx_s) \in A_{i+1}^s$ is a solution to \eqref{multi equation}, then the tuple $(\vr + q \vx_1, \dots, \vr + q \vx_s)$ is a solution to \eqref{multi equation} in $A_i^s$.  Since $A_i$ has only diagonal solutions to \eqref{multi equation}, it follows that $(\vx_1, \dots, \vx_s)$ is itself diagonal.   Assuming \eqref{second X_i lower bound}, we have therefore obtained another quadruple $(Q_{i+1}, A_{i+1}, X_{i+1}, \delta_{i+1})$ satisfying conditions (i) to (v).

As long as \eqref{second X_i lower bound} holds, we can iterate this construction.  After $\ceil{c^{-1} \delta^{1-s}}$ such iterations we have a density $\delta_i$ of size at least $2 \delta$.  After a further $\ceil{ c^{-1} (2\delta)^{1-s}}$ such iterations, we have a density of at least $4 \delta$.  Thus, setting $L = \floor{\log_2 (\delta^{-1})}$, we see that after a total of
$$
I = \sum_{l =0}^L \ceil{ c^{-1} (2^l\delta)^{1-s}}
$$
iterations, we have a density of $2^{L+1} \delta > 1$ (a contradiction).  Hence \eqref{second X_i lower bound} cannot hold for all $0 \leq i \leq I$.  Thus for some $i \in \set{0, 1, \dots, I}$ we have
\begin{equation}\label{delta I lower bound}
\begin{split}
\exp(C/\delta)  \geq \exp(C/\delta_i) & \geq X_i\\
							& \geq X_{i-1}^\tau 2^{-k} \geq X_{i-2}^{\tau^2} 2^{-k(1 + \tau)} \geq \dots \geq X_0^{\tau^i} 2^{-k/(1-\tau)}\\
							& \geq X^{\tau^i} 2^{-k/(1-\tau)}.
\end{split}
\end{equation}
Taking logarithms in \eqref{delta I lower bound}, we therefore have
$$
C/\delta \geq \tau^i \log X - \tfrac{k}{1- \tau}.
$$
Notice that $i \leq I \leq 2 c^{-1} \delta^{1-s}$.  So on taking logarithms again we  have
\begin{equation}\label{final almost bound}
\log(C\delta^{-1}+k(1-\tau)^{-1}) + 2 c^{-1} \delta^{1-s} \log (1/\tau) \geq \log\log X.
\end{equation}
Crude estimation shows that the left hand side of \eqref{final almost bound} is $O_{k, \Phi}(\delta^{1-s})$, as required.
\end{proof}

We begin the proof of Lemma \ref{increment lemma} with the following general result on partitioning phase polynomials into approximate level sets.

\begin{lemma}\label{small diameter lemma}  Let $P(x_1,x_2)$ denote a real polynomial of degree $k$.  Set 
$$
\tau_k^{-1} = 24^k (k!)^2 2^{k(k+1)/2}.
$$
There exists a positive constant $C = C(k)$, such that for any half-open square $Q$ of side-length $X$, we can find half-open squares $Q_1, \dots, Q_n$ each with side-length at least $ 2^{-k}X^{\tau_k}$, along with $q_i \in \N$ and $\vr_i \in \Z^2$, such that the sets $\vr_i + q_i \cdot [Q_i]$ partition $[Q]$, and furthermore for any $\vx, \vy \in \vr_i + q_i \cdot [Q_i]$ we have
\begin{equation}
\norm{ P(\vx) - P(\vy) } \leq CX^{-\tau_k}.
\end{equation}
\end{lemma}

\begin{proof}
By Taylor's formula
\begin{equation}\label{another taylor}
\begin{split}
P(\vr + q \vx) = &\sum_{u,v \geq 0} \frac{r_1^u r_2^v}{u! v!} P^{u,v}(q\vx)\\
			& = q^k F(\vx) + G(\vx; \vr, q),
\end{split}
\end{equation}
where $F(\vx)$ is a homogeneous real polynomial of degree $k$, and $G(\vx) = G(\vx; \vr, q)$ is a polynomial of degree strictly less than $k$.  Set
\begin{equation}
F(\vx) = \sum_{0 \leq l \leq k} \alpha_l x_1^l x_2^{k-l},
\end{equation}
and let $\sigma_k^{-1} = 6 k 2^k$.  It follows from Baker \cite[Theorem 8.1]{baker86} that there exists $C_1 = C_1(k)$ such that for any $Y \geq 1$ there is some some $1 \leq q \leq Y$ satisfying
\begin{equation}\label{simul dioph approx}
\bignorm{q^k \alpha_l} \leq C_1Y^{- \sigma_k} \quad (0 \leq l \leq k).
\end{equation}
Let $Y = X^{1/2}$ in \eqref{simul dioph approx}, where $X$ is the side-length of $Q$. Partitioning $[Q]$ into congruence classes modulo $q$, we have
$$
[Q] = \bigcup_{\vr \in [q]^2} \vr + q \cdot [Q(\vr)],
$$
where $Q(\vr)$ is a half-open square of side-length $X/q$.  Let us set
$$
t =\ceil{ q^{-1} X^{1- \frac{\sigma_k}{4k}}}.
$$
Then we can partition each $Q(\vr)$ into $t^2$ half-open squares $Q(\vr, \vt)$ ($\vt \in [t]^2$), each of side-length $X/(qt)$.  For fixed $\vr$ and $\vt$ let us pick $\va(\vr, \vt) \in [Q(\vr, \vt)]$.  Then we have $[Q(\vr, \vt)] = \va(\vr, \vt) + [Q'(\vr, \vt)]$, where $Q'(\vr, \vt)$ is a half-open square of side-length $X/(qt)$ satisfying
\begin{equation}
Q'(\vr, \vt) \subset [-X/(qt), X/(qt)]^2.
\end{equation}
  It follows that there exist pairs $\vb(\vr, \vt) \in \Z^2$ ($\vr \in [q]^2$, $\vt \in [\vT]$) such that the set $[Q]$ is equal to the disjoint union
$$
\bigcup_{\vr \in [q]^2} \bigcup_{\vt \in [\vT]} \bigbrac{ \vb(\vr, \vt) + q\cdot [Q'(\vr, \vt)]}.
$$

Clearly $X/(qt) \leq X^{\sigma_k/(4k)}$.  Since $q \leq X^{1/2} \leq X^{1 - \sigma_k/(4k)}$, we also have 
$$
X/(qt) \geq \frac{X}{X^{1- \sigma_k/(2k)} + q} \geq \trecip{2} X^{\sigma_k/(4k)} .
$$
Hence the side-length of each $Q'(\vr, \vt)$ is between $\trecip{2}X^{\sigma_k/(4k)}$ and $X^{\sigma_k/(4k)}$.  It follows that for any $\vx, \vy \in [Q'(\vr, \vt)]$ we have
\begin{equation}\label{F dioph ineq}
\bignorm{q^k(F(\vx) - F(\vy))}   \ll_k X^{-\sigma_k/4}.
\end{equation}

Write $G_{\vr, \vt}(\vx)$ for the polynomial $G(\vx; \vb(\vr, \vt), q)$.  By induction, there exists a partition of $[Q'(\vr, \vt)]$ into sets of the form $\vs_i + q_i \cdot [Q_i'(\vr, \vt)]$ ($1 \leq i \leq m = m(\vr, \vt)$), where $Q_i'(\vr, \vt)$ is a half-open square of side-length at least 
$$
2^{1-k-\tau_{k-1}}X^{\sigma_k\tau_{k-1}/(4k)},
$$
and such that for any $\vx, \vy \in [Q_i'(\vr, \vt)]$ we have
\begin{equation}\label{G dioph ineq}
\norm{G_{\vr, \vt}(\vs_i + q_i\vx) - G_{\vr, \vt}(\vs_i + q_i \vy)} \ll_k X^{-\sigma_k\tau_{k-1}/(4k) }.
\end{equation}

Let us write $q'_i(\vr, \vt)$ for $q q_i(\vr, \vt)$ and $\vb'_i(\vr, \vt)$ for $\vb(\vr, \vt) + q \vs_i(\vr, \vt)$.
Then $[Q]$ is partitioned by the sets
$$
\vb'_i(\vr, \vt) + q'_i(\vr, \vt) \cdot [Q_i'(\vr, \vt)] \quad (\vr \in [q]^2,\ \vt \in [\vT],\ 1 \leq i \leq m(\vr, \vt)).
$$
Combining \eqref{another taylor}, \eqref{F dioph ineq} and \eqref{G dioph ineq}, we see that for each $\vx, \vy \in  [Q_i'(\vr, \vt)]$ we have, on writing $\vb'= \vb'_i(\vr, \vt)$ and $q' = q'_i(\vr, \vt)$, that
$$
\norm{P(\vb'+q'\vx) - P(\vb' + q'\vy)} \ll X^{-\sigma_k \tau_{k-1}/(4k)}.
$$
A simple calculation reveals that $\sigma_k \tau_{k-1}/(4k) = \tau_k$, as required.
\end{proof}

\begin{proof}[Proof of Lemma \ref{increment lemma}]
Let us define 
$$
f_A(\valpha) = \sum_{\vx} 1_{A}(\vx) e(\valpha\cdot\vF(\vx)),
$$
together with $f(\valpha) = f_{[Q]}(\valpha)$ and $g_A(\valpha) = f_A(\valpha) - \delta f(\valpha)$.  Using translation invariance of \eqref{multi equation}, we have $R_{\vc, \Phi}([Q]) \sim R_{\vc, \Phi}(X)$.  Let $c_1$ equal the quantity  
$
\sigma_\infty(\vc) \prod_p \sigma_p(\vc)
$
from Theorem \ref{asymptotic formula}.  Provided we take $C$ in Lemma \ref{increment lemma} sufficiently large, so that $X \geq \exp(C/\delta^{-1})\gg_{\vc, \Phi} 1$, we can use Theorem \ref{asymptotic formula} to ensure that 
\begin{equation}\label{uniform lower bound}
R_{\vc, \Phi}([Q]) \geq \trecip{2}c_1X^{2s - K}.
\end{equation}
The assumption that $\vc$ is a non-singular choice for $\Phi$ implies that $c_1$ is positive.  Let us take $c$ in Lemma \ref{increment lemma} sufficiently small, say $c \leq \trecip{4} c_1$.  Combining \eqref{A conditions}, \eqref{uniform lower bound}, orthogonality and H\"{o}lder's inequality, we have
\begin{align*}
\trecip{4}c_1\delta^s X^{2s - K}& \leq |R_{\vc, \Phi}(A) - \delta^s R_{\vc, \Phi}([Q])|\\
				& \leq \oint |f_A(c_1\valpha) \dotsm f_A(c_s\valpha) - \delta^s f(c_1\valpha) \dotsm f(c_s\valpha)| \intd \valpha\\
				& \leq \sup_{\valpha}|g_A(\valpha)| X^{2\epsilon} \oint \bigbrac{ |f_A(\valpha)|^{2t} + |f(\valpha)|^{2t}} \intd\valpha,
\end{align*}
where $s = 1 + \epsilon + 2t$, for some $\epsilon \in \set{0,1}$.  Since $2t \geq kN(\log K + \log \log K + 26)$, we can use Theorem \ref{asymptotic formula} (together with the underlying Diophantine equation), to conclude that
\begin{align*}
\oint |f_A(\valpha)|^{2t} \intd \valpha & \leq \oint |f(\valpha)|^{2t} \intd\valpha\\
							& \ll_{\vc, \Phi} X^{4t - K}.
\end{align*}

Setting $b_A(\vx) = 1_A(\vx) - \delta 1_{[Q]}(\vx)$, we see that there exists $\valpha \in \T^N$ such that 
\begin{equation}\label{balanced correlation}
\Bigabs{ \sum_{\vx \in [Q]^2} b_A(\vx) e(\valpha \cdot \vF(\vx))} = |g_A(\valpha)| \gg_{\vc, \Phi} \delta^s X^2.
\end{equation}
Let $\tau = \tau_k$ and $C_1 = C_1(k)$ be as in Lemma \ref{small diameter lemma}, and consider the polynomial $P(\vx) = \valpha \cdot \vF(\vx)$.  Then there exist half-open squares $Q_1, \dots, Q_n$ each of side-length at least $2^{-k} X^\tau$, along with $\vr_i$ and $q_i$ ($1 \leq i \leq n$) such that the sets $\vr_i + q_i\cdot [Q_i]$ partition $[Q]$, and for any $\vx, \vy \in \vr_i + q_i\cdot [Q_i]$ we have $\norm{P(\vx) - P(\vy)} \leq C_1 X^{-\tau}$.  Notice that this implies that $|e(P(\vx) - e(P(\vy))| \ll X^{-\tau}$.  Thus
\begin{align*}
\biggabs{ \sum_{\vx \in [X]^2} b_A(\vx) e(\valpha \cdot \vF(\vx))} & \leq \sum_{i=1}^n\biggabs{ \sum_{\vx \in \vr_i + q_i\cdot [Q_i]} b_A(\vx) e(P(\vx))} \\
											& = \sum_{i=1}^n\biggabs{ \sum_{\vx \in \vr_i + q_i\cdot [Q_i]} b_A(\vx)} + O\bigbrac{X^{2-\tau}}. 
\end{align*}
We can take $C$ in Lemma \ref{increment lemma} sufficiently large to ensure that the lower bound $X \geq \exp(C/\delta)$ implies that the $O(X^{2-\tau})$ term above is at most half the size of the right hand side of \eqref{balanced correlation}.  We thereby obtain that
\begin{equation}\label{balance large absolute}
\sum_{i=1}^n\biggabs{ \sum_{\vx \in \vr_i + q_i\cdot [Q_i]} b_A(\vx)} \gg_{\vc,\Phi} \delta^s X^2.
\end{equation}
Let $\mathcal{I}$ denote the set of $i \in [n]$ for which $\sum_{\vx \in \vr_i + q_i\cdot [Q_i]} b_A(\vx) \geq 0$.  Since $b_A$ has average zero, we can add $\sum_{\vx} b_A(\vx)$ to the left side of \eqref{balance large absolute}, to obtain
\begin{equation}
\sum_{i\in \mathcal{I}} \Bigbrac{\sum_{\vx \in \vr_i + q_i \cdot [Q_i]} b_A(\vx)} \gg_{\vc, \Phi} \delta^s X^2.
\end{equation}
The density increment \eqref{square increment} now follows from the pigeon-hole principle, provided we take $c = c(\vc, \Phi)$ sufficiently small.
\end{proof}

Our originally advertised theorem, Theorem \ref{intro result 1}, now almost follows.  It  remains to show that $kN(\log K + \log \log K +27) \leq \tfrac{3}{4} k^3 \log k (1 + o(1))$.  We have the trivial bound $N \leq k^2$ and $K \leq kN \leq k^3$.  The bound $N \leq \tfrac{k^2}{4} (1 + o(1))$ takes a little more calculation, but follows from the fact that the number of linearly independent derivatives $\Phi^{u,v}$ with $u+ v = k - d$ is at most $\max\set{k+1 - d, d+1}$.

\begin{acknowledgements} The author would like to express his gratitude to Professor Wooley for his unending encouragement, patience and generosity with ideas, and Professor Parsell for his insights into \S \ref{Weyl-type section}. 
\end{acknowledgements}

 \bibliographystyle{amsplain}
 \bibliography{/Users/seanprendiville/Dropbox/Maths_Latex_Docs/Bibliography}

\end{document}